\documentclass[oneside,english]{amsart}

\usepackage[T1]{fontenc}
\usepackage[latin9,utf8]{inputenc}
\usepackage{geometry}
\usepackage{babel}
\usepackage{amsbsy}
\usepackage{amstext}
\usepackage{amsmath}
\usepackage{amsthm}
\usepackage{amssymb}
\usepackage{graphicx}
\usepackage{verbatim}
\usepackage{subfigure}
\usepackage{caption}
\usepackage[all]{xy}
 \usepackage{hyperref}
\usepackage{xcolor}
\usepackage{soul}
\usepackage[shortlabels]{enumitem}
\makeatletter

\usepackage{cleveref}
\usepackage[normalem]{ulem}

\numberwithin{equation}{section}
\numberwithin{figure}{section}
\theoremstyle{plain}
\newtheorem{thm}{\protect\theoremname}[section]
\theoremstyle{plain}
\newtheorem{cor}[thm]{\protect\corollaryname}
\theoremstyle{plain}
\newtheorem{lem}[thm]{Lemma}
\theoremstyle{plain}
\newtheorem{prop}[thm]{Proposition}
\theoremstyle{definition}

\theoremstyle{definition}
\newtheorem{defn}[thm]{\protect\definitionname}
\theoremstyle{definition}
\newtheorem{exmp}[thm]{Example}
\theoremstyle{definition}
\newtheorem{qn}{Question}
\newtheorem{rem}[thm]{\protect\remarkname}

\newcommand{\p}{\mathbf{p}}
\newcommand{\q}{\mathbf{q}}
\newcommand{\bfl}{\mathbf{l}}

\newcommand{\uCFK}{CFK^{\prime}}
\newcommand{\uHFK}{HFK^{\prime}}
\newcommand{\uCFL}{CFL^{\prime}}

\newcommand{\uCFKI}{CFKI^{\prime}}
\newcommand{\uHFKI}{HFKI^{\prime}}
\newcommand{\ubar}[1]{\text{\b{$#1$}}}

\title{On the nonorientable four-ball genus of torus knots}
\author{Fraser Binns}
\address{Department of Mathematics, Princeton University}
\email{fb1673@prinnceton.edu}
\author{Sungkyung Kang}
\address{Mathematical Institute, University of Oxford}
\email{skkang@ibs.re.kr}
\author{Jonathan Simone}
\address{School of Mathematics, Georgia Institute of Technology}
\email{jsimone7@gatech.edu}
\author{Paula Tru\"{o}l}
\address{Max Planck Institute for Mathematics}
\email{truoel@mpim-bonn.mpg.de}

\makeatother

\providecommand{\corollaryname}{Corollary}
\providecommand{\definitionname}{Definition}
\providecommand{\remarkname}{Remark}
\providecommand{\theoremname}{Theorem}

\newcommand{\ie}{i.\,e.~}
\newcommand{\eg}{e.\,g.~}
\newcommand{\QQ}{\mathbb{Q}}
\DeclareMathOperator{\arf}{Arf}
\newcommand{\gammatop}{\gamma_4^{\text{top}}}

\usepackage{mathtools}
 
\usepackage[colorinlistoftodos,prependcaption,textsize=tiny]{todonotes}
\usepackage{xargs} 
\newcommandx{\info}[2][1=]{\todo[linecolor=green,backgroundcolor=green!25,bordercolor=green,#1]{#2}}

\begin{document}

\begin{abstract}
The nonorientable four-ball genus of a knot $K$ in $S^3$ is the minimal first Betti number of nonorientable surfaces in $B^4$ bounded by $K$. By amalgamating ideas from involutive knot Floer homology and unoriented knot Floer homology, we give a new lower bound on the smooth nonorientable four-ball genus $\gamma_4$ of any knot. This bound is sharp for several families of torus knots, including $T_{4n,(2n\pm 1)^2}$ for even $n\ge2$, a family Longo showed were counterexamples to Batson's conjecture. We also prove that, whenever $p$ is an even positive integer and $\frac{p}{2}$ is not a perfect square, the torus knot $T_{p,q}$ does not bound a locally flat M\"obius band for almost all integers $q$ relatively prime to $p$.
\end{abstract}
\maketitle

\section{Introduction}
The \emph{nonorientable (smooth) $4$-ball genus} $\gamma_4(K)$ of a knot $K$ in the $3$-sphere $S^3$ is the minimal first Betti number $b_1(F) = \dim H_1(F; \QQ)$ of any smoothly embedded nonorientable surface $F$ in the $4$-ball $B^4$ with $\partial F = K$. 
By definition, $\gamma_4(K) \geq 1$ for any knot $K$ in $S^3$, and $\gamma_4(K) = 1$ if and only if $K$ bounds a smooth M{\"o}bius band in $B^4$. Moreover, if $K$ is smoothly slice, \ie $K$ bounds a smoothly embedded disk in $B^4$, then $\gamma_4(K)=1$.  

First introduced by Murakami and Yasuhara \cite{murakami2000four}, $\gamma_4$ has drawn increased attention in recent years. 
One class of knots that has received considerable attention is the class of torus knots $T_{p,q}$ for relatively prime positive integers $p$ and $q$.
For $p\le 3$ (or similarly, $q\le 3$), it is known that $\gamma_4(T_{p,q})=1$: $T_{1,q}$ is the unknot for all $q$; $T_{2,q}$ bounds an obvious band with $q$ half twists (\ie a M{\"o}bius band); and a single nonorientable band move turns $T_{3,q}$ into the unknot for each $q$ (see \cite[Section 2.1]{allen2020nonorientable}). 
In general, however, the values of $\gamma_4$ for torus knots are not known.

In \cite{batson2012nonorientable}, Batson constructed a nonorientable surface $F_{p,q}$ in $B^4$ with $\partial F_{p,q} =T_{p,q}$ via \emph{pinch moves} (defined in \Cref{sec:upperBounds}). The  minimal first Betti number of $F_{p,q}$ is called the \emph{pinch number} of $T_{p,q}$, which we denote by $\vartheta(T_{p,q})$. Batson conjectured that $\gamma_4(T_{p,q})=\vartheta(T_{p,q})$. This could be seen as the nonorientable analogue of the Milnor conjecture
\begin{align}\label{eq:orientablegeneraoftorusknots}
g_4\left(T_{p,q}\right) = g\left(T_{p,q}\right) = \frac{(p-1)(q-1)}{2},
\end{align}
 which was first proven by Kronheimer and Mrowka \cite{kronheimermrowka}. Here $g_4(K)$ denotes the \emph{(orientable) $4$-ball genus} of a knot $K$ in $S^3$ --- the minimal genus of a compact, oriented surface smoothly embedded in $B^4$ bounded by $K$ --- and $g(K)$ denotes the \emph{$3$-genus} of a knot $K$ in $S^3$ --- the minimal genus of a compact, oriented surface smoothly embedded in $S^3$ bounded by $K$.\footnote{In contrast to the orientable setting, where $g_4\left(T_{p,q}\right)= g\left(T_{p,q}\right)$ (see \Cref{eq:orientablegeneraoftorusknots}), it is in general not true that $\gamma_4\left(T_{p,q}\right) = \gamma_3\left(T_{p,q}\right)$ for any relatively prime positive integers $p$ and $q$, where $\gamma_3\left(T_{p,q}\right)$ is the minimal first Betti number of any nonorientable smooth surface $F$ in $S^3$ with $\partial F = K$ (as defined by Clark \cite{clark}).
In fact, the difference between $\gamma_3$ and $\gamma_4$ can be arbitrarily large \cite[Theorem 1.1]{jabukavanCott3and4}.
By work of Teragaito \cite{teragaito}, it is known that the value of $\gamma_3(T_{p,q})$ is determined by a recursive, arithmetic function of $p$ and $q$.}

By showing that $\gamma_4\left(T_{2k,2k-1}\right)= \vartheta\left(T_{2k,2k-1}\right)= k-1$ for all $k \geq 1$, Batson provided an infinite family of torus knots for which his conjecture is true. Further such infinite families of torus knots 
were found by Jabuka and Van Cott \cite{jabuka2019nonorientable}. 
However, Batson's conjecture was disproved by Lobb \cite{lobb}, who found that $\gamma_4(T_{4,9})=1$ while $\vartheta(T_{4,9})=2$. Soon after, Longo \cite{longo} generalized this example, showing that $\gamma_4(T_{4n,(2n\pm 1)^2})\leq 2n-1$ while $\vartheta(T_{4n,(2n\pm 1)^2})=2n$ for all $n\ge 2$. Tairi \cite{tairi} provided another infinite family of counterexamples, showing in particular that $\gamma_4 (T_{4,11})=1$ while $\vartheta(T_{4,11})=2$.

In this article, we will shed more light on the nonorientable $4$-ball genus of torus knots, in both the smooth and topological categories. In particular, we will:
\begin{itemize}
    \item develop a new lower bound on $\gamma_4$ using involutive unoriented Floer homology (\Cref{mobiusthm} and \Cref{highergenusthm} below, see \Cref{subsec:lowerbound} for details), which --- for so-called $L$-space knots (and consequently, torus knots) --- can be read off from the Alexander polynomial (\Cref{thm:lowerbound} below);
    \item develop an obstruction to torus knots bounding (locally flat) M{\"o}bius bands using linking forms and number-theoretic arguments (see \Cref{subsec:obstructmobius});
    \item and use these new tools along with existing tools to either compute or give narrow bounds on $\gamma_4$ for some infinite families of torus knots.
\end{itemize}  

Throughout, we will always assume that for any torus knot $T_{p,q}$, $p$ and $q$ are positive, relatively prime integers. The main lower bound on $\gamma_4$ that we will use  is the following.

\begin{thm}
Let $K$ be an $L$-space knot. Write the Alexander polynomial $\Delta_K(t)$ of $K$ as a monic, symmetric Laurent polynomial, and suppose that its constant term is $-1$ and that the first nonzero term of positive degree is $t^k$, \ie $\Delta_K(t)=\cdots + t^{-k} - 1 + t^k - \cdots$. 
Then we have the following lower bound on the nonorientable $4$-ball genus:
$k-1 \leq \gamma_4(K)$. In addition, if $K$ bounds a smoothly embedded M{\" o}bius band in $B^4$, then $k=1$. 
\label{thm:lowerbound}
\end{thm}

\begin{rem}
To the best of the authors' knowledge, \Cref{thm:lowerbound} provides the first Heegaard-Floer theoretical proof that the figure-eight knot does not bound a smooth Möbius band in $B^4$ (see \Cref{ex:figeight}). This seems worth mentioning since it was not until the advent of involutive theory that Heegaard Floer could be used to show that the figure-eight knot is not slice \cite{hendricks2017involutive}. 
\end{rem}

\begin{rem} Notice the difference between the M{\" o}bius band case and the case of higher first Betti number. Interestingly, this difference is substantial and cannot be removed since the bound in the case of higher first Betti number is sometimes sharp; see the proof of \Cref{prop:longoextension} in \Cref{sec:calculations}.
\label{rem:difference}
\end{rem}

\begin{rem}
The integer $k$ in \Cref{thm:lowerbound} is called the \emph{stretch} of the given $L$-space knot $K$, which was defined by Borodzik and Hom in \cite{borodzik2019involutive}. It is shown in \cite[Theorem 3.10]{borodzik2019involutive} that the stretch of a torus knot $T_{p,q}$ is given by $\left \lfloor{\frac{a_k-1}{2}}\right \rfloor +1$, where $q>p$ and $$\frac{q}{p}=[a_0,\cdots,a_k]=a_0 +\frac{1}{a_1 +\displaystyle \frac{1}{\displaystyle a_2+\frac{1}{\ddots+\displaystyle\frac{1}{a_k}}}}$$ is a regular continued fraction expansion with $a_k >1$ and $k\ge 2$.
\end{rem}

A quick overview of the proof of \Cref{thm:lowerbound} is provided in \Cref{subsec:lowerbound}, while the full proof can be found in \Cref{sec:involutive}. It turns out that \Cref{thm:lowerbound} provides a sharper bound than many other existing lower bounds involving knot invariants (see \Cref{sec:background}). In fact, we will use \Cref{thm:lowerbound} to give a partial affirmative answer to a question of Longo \cite[Question 4.3]{longo}, which asks whether his examples $T_{4n,(2n\pm 1)^2}$ satisfy  $\gamma_4(T_{4n,(2n\pm1)^2})=2n-1$.

\begin{prop}
If $n\ge 2$ is even, then $\gamma_4\left(T_{4n,(2n\pm1)^2}\right)=2n-1=\vartheta\left(T_{4n,(2n\pm 1)^2}\right) -1$.
\label{prop:longoanswer}
\end{prop}

In fact, we prove the following more general result.

\begin{prop}
Let $n\ge2$ be even and $k\ge 0$. Then\label{prop:longoextension} $$\gamma_4\left(T_{4n+2k,(4n+2k)(n\pm1)+1}\right) = 2n+k-1=\vartheta\left(T_{4n+2k,(4n+2k)(n\pm1)+1}\right)-1.$$
\end{prop}

Using \Cref{thm:lowerbound} and other existing results (see Section \ref{sec:background}), we can also find narrow bounds on $\gamma_4$ for certain infinite families of torus knots $T_{p,q}$ with fixed $p$.

\begin{thm} Fix $p>3$. If $q\equiv p-1, p+1,\text{ or }2p-1\pmod{2p},$ then $ \vartheta(T_{p,q})-1\le\gamma_4(T_{p,q})\le\vartheta(T_{p,q}).$ 

In particular, we have the following.
\begin{enumerate}[(a)]
    \item\label{item:p-1} Let $p$ be odd. 
    
    If $q\equiv p-1\pmod{2p},$ then $\gamma_4(T(p,q))\in\left\{\frac{p-3}{2},\frac{p-1}{2}\right\}$. 
    
    If $q\equiv p+1 \text{ or }2p-1 \pmod{2p}$, then $\gamma_4(T_{p,q})=\frac{p-1}{2}$. 
    \item \label{item:c} Let $p$ be even. 
    
    If $q>p$ and $q\equiv p-1, p+1\text{ or }2p-1\pmod{2p}$, then $\gamma_4\left(T_{p,q}\right)\in\left\{\frac{p-2}{2},\frac{p}{2}\right\}$.
    
    \end{enumerate}
\label{thm:evencalculations}
\end{thm}

\begin{rem}
Notice that the torus knots in \Cref{prop:longoextension} belong to \Cref{thm:evencalculations}\ref{item:c}. In particular, \Cref{thm:evencalculations}\ref{item:c} tells us that $\gamma_4(T_{4n+2k,(4n+2k)(n\pm1)+1}) \in\{ 2n+k-1, 2n+k\}$. To prove the first equality in \Cref{prop:longoextension}, 
we will show that performing $k$ pinch moves on $T_{4n+2k,(4n+2k)(n\pm1)+1}$ yields $T_{4n,(2n\pm1)^2}$ and then apply \Cref{prop:longoanswer}.
\label{rem:specialcase}
\end{rem}

\begin{rem}
The lower bound on $\gamma_4$ given in the first part of \Cref{thm:evencalculations}\ref{item:p-1} is obtained using Theorem \ref{thm:lowerbound}. We
will see that other existing lower bounds involving knot invariants (see Section \ref{sec:background}) do not suffice to show this.
The second statement of \Cref{thm:evencalculations}\ref{item:p-1} follows from existing work (see Corollary \ref{cor:jabukavanCottpinch}). Similarly, the middle case of \Cref{thm:evencalculations}\ref{item:c} is proved using Theorem \ref{thm:lowerbound}, whereas the others follow from existing work (see Corollary \ref{cor:jabukavanCottpinch}).
\label{newandexisting}\end{rem}

\begin{rem} Consider the family of torus knots in Theorem \ref{thm:evencalculations}\ref{item:c} with $q=p+1$. These are of the form $T_{2n,2n+1}$. In \cite[Proposition 5.7]{feller2020nonorientable} it was shown that these torus knots do not bound a smooth M{\" o}bius band in $B^4$ if $n \equiv 3 \mod 4$, $2n+1$ is square-free, and $n+1$ is not a square. \Cref{thm:evencalculations}\ref{item:c} generalizes this result.
\end{rem}

\begin{rem}
Notice that \Cref{thm:evencalculations} includes the case $q=p-1$ when $p$ is odd, but not when $p$ is even. The latter case was solved by Batson \cite{batson2012nonorientable}; he showed that $\gamma_4\left(T_{p,p-1}\right)=\frac{p-2}{2}=\vartheta(T_{p,p-1})$ for even $p$.
\end{rem}

We now turn our attention to the family of torus knots $T_{4,q}$ for odd $q > 4$. 
We note that the bound, $2$, that \Cref{thm:evencalculations}\ref{item:c} yields is no better than the one obtained by purely diagrammatic methods in Lemma~\ref{pinchlemma}. However, we are able to get more mileage out of \Cref{thm:lowerbound} than was used in \Cref{thm:evencalculations}\ref{item:c} since we need only obstruct the existence of M{\"o}bius bands (c.f. Remark \ref{rem:difference}).

\begin{thm} 
If $q\equiv 5\text{ or }7\pmod 8$, then $\gamma_4(T_{4,q})=2$. Moreover, $T_{4,q}$ does not bound a M{\" o}bius band for almost all $q\equiv 1\text{ or }3\pmod{8}$.
\label{thm:4q}
\end{thm}

In \Cref{thm:4q}, since $q$ is relatively prime to 4, there are four cases to consider: $q\equiv 1,3,5,7\pmod 8$. The cases $q\equiv 5,7\pmod 8$ are easy to deal with: we first show that $\gamma_4(T_{4,q})\le\vartheta(T_{4,q})=2$; and then we use \Cref{thm:lowerbound} to show that $\gamma_4(T_{4,q})=2$. When $q\equiv 1,3 \pmod 8$, we again have that $\gamma_4(T_{4,q})\le\vartheta(T_{4,q})=2$, but the lower bound on $\gamma_4$ given by \Cref{thm:lowerbound} does not apply. Moreover,  other existing lower bounds involving knot invariants (see \Cref{sec:background} and \Cref{subsec:4q}) do not provide any additional information.
At this point, one might wonder whether \Cref{thm:4q} can be upgraded to $\gamma_4(T_{4,q})=2$ for all $q\equiv 1,3\pmod 8$, confirming Batson's conjecture for the torus knots $T_{4,q}$. It is known, however, that this is not the case --- as mentioned above, we have $\gamma_4(T_{4,9})=1$ \cite{lobb} and $\gamma_4(T_{4,11})=1$ \cite{tairi}. 

The issue with applying \Cref{thm:lowerbound} and the other existing lower bounds is that these lower bounds all use knot invariants  (\eg the Arf invariant, signature, upsilon invariant, etc.) that are periodic modulo $2p$ for $T_{p,q}$. Since $\gamma_4(T_{p,1})=\gamma_4(T_{p,3})=1$ for any $p$, this periodic behavior implies that all of these lower bounds on $\gamma_4(T_{p,q})$ are at most 1 for all $q\equiv 1\text{ or }3\pmod{2p}$. In particular we cannot use these lower bounds to obstruct $T_{4,q}$ from bounding a M{\"o}bius band for $q\equiv 1\text{ or }3\pmod 8$. Instead, to prove the second statement of Theorem \ref{thm:4q}, we consider the linking form of the double cover of $S^3$ branched over $T_{4,q}$ (and more generally branched over $T_{p,q}$ for $p$ even) and apply \Cref{asymptotic} below. We in fact prove something stronger, namely that these torus knots do not even bound \emph{locally flat} M{\"o}bius bands. This gives us information about the \emph{topological nonorientable $4$-ball genus} $\gammatop$ of $T_{4,q}$.  It is worth noting that this obstruction does not work if $p$ and $q$ are both odd; in this case, the double cover of $S^3$ branched over $T_{p,q}$, which is the Brieskorn sphere $\Sigma(2,p,q)$, is an integral homology sphere (as opposed to a rational homology sphere if $p$ or $q$ is even) and thus the linking form obstruction vanishes. 

\subsection{Obstructing Möbius bands in the topological category}\label{subsec:obstructmobius}

For a knot $K$ in $S^3$, let $\gammatop(K)$ denote the minimal first Betti number of any locally flat nonorientable surface in $B^4$ bounding $K$; note that $\gammatop(K) \leq \gamma_4(K)$ for any knot $K$. 
Gilmer and Livingston \cite[Theorem D]{gilmerlivingston} showed that there exist knots $K$ with $\gammatop(K)=1$, but $\gamma_4(K) > 1$. They found these knots by taking the connected sum of a topologically slice knot with a knot $K$ satisfying $\gammatop(K) = \gamma_4(K) =1$. In fact, the nonorientable $4$-ball genus of topologically slice knots can be arbitrarily large \cite{fellerparkray}.

In \cite[Theorem 1.3]
{feller2020nonorientable}, Feller and Golla showed that  $\gamma_4^{top}(T_{2n-1, 2n}) \leq n-2 < n-1= \gamma_4(T_{2n-1, 2n})$ for all $n \geq 5$; the latter equality is due to Batson \cite{batson2012nonorientable}. Moreover, they provided new infinite families of knots not bounding locally flat Möbius bands and were particularly interested in the question of whether there exists a torus knot $T_{p,q}$ with $\gammatop(T_{p,q})=1$ and $\gamma_4(T_{p,q}) > 1$. To the best of the authors' knowledge, this question still remains open.

In \Cref{sec:asymp}, using a bound of Murakami and Yasuhara \cite[Theorem 2.5]{murakami2000four} we prove that torus knots which bound locally flat M{\" o}bius bands are very rare: when $p$ is even and $\frac{p}{2}$ is not a square, the set of odd integers $q$ satisfying $\gammatop(T_{p,q})=1$ has ``zero density''. In particular, we prove the following theorem.

\begin{thm}
\label{asymptotic}
Let $p$ be a positive even integer such that $\frac{p}{2}$ is not a perfect square. Given a positive integer $N$, consider the set $S_{N,p}=\{n\in \mathbb{Z}\,\vert\,1\le n\le N,\,\gcd(p,n)=1\}$. Then 
\[
\lim_{N\rightarrow\infty} \frac{\sharp\left\{q\in S_{N,p} \,\vert\, T_{p,q}\text{ bounds a locally flat M{\" o}bius band in }B^4\right\}}{\sharp S_{N,p}} = 0.
\]
\end{thm}

\begin{rem}
The proof of \Cref{asymptotic} actually shows the following: if $p$ is even and $\frac{p}{2}$ is not a perfect square, then there exists an integer $r$, relatively prime to $p$, such that if $q$ admits a prime factor $s$ with odd exponent in the prime factorization of $q$ and $s \equiv r \pmod{2p}$,
then $T_{p,q}$ does not bound a locally flat Möbius band in $B^4$ (see \Cref{torusobstlem}). For example, if $q\equiv 1\text{ or }3\pmod 8$ and $q$ has a prime factor $s\equiv 5\pmod 8$ such that the exponent of $s$ in the prime factorization of $q$ is odd, then $T_{4,q}$ does not bound a locally flat Möbius band (see Example \ref{exmp:4q} for details). Notice that the smallest values of $q$ satisfying this condition are $q=35, 65, 91, 105$. Thus $T_{4,35}$, $T_{4,65}$, $T_{4,91}$, and $T_{4,105}$ do not bound locally flat M{\"o}bius bands; moreover, for any $q\equiv 1\text{ or }3\pmod 8$ with $q<105$ and $q\neq 35,65,91$, it remains unknown whether the torus knot $T_{4,q}$ bounds a locally flat M{\"o}bius band.\label{rem:primefactors}
\end{rem}

\subsection{New lower bound on the smooth nonorientable 4-ball genus}\label{subsec:lowerbound}
In the smooth category, our new lower bound (see \Cref{thm:lowerbound}) comes from a construction which amalgamates involutive Heegaard Floer homology \cite{hendricks2017involutive} and unoriented knot Floer homology \cite{fan2019unoriented}. Given a knot $K$ in $S^3$, we construct a module over the ring $\mathbb{F}_{2}[U,Q]/\left(Q^2\right)$, which we call the \emph{involutive unoriented knot Floer homology} $\uHFKI(K)$. Using the structure of $\uHFKI(K)$, we can define two concordance invariants $\bar{\upsilon}$ and $\ubar{\upsilon}$, which are the values at $1$ of the \emph{involutive Upsilon invariants} $\bar{\Upsilon}_t$ and $\ubar{\Upsilon}_t$ defined by Hogancamp and Livingston \cite{hogancamp2017involutive}.

Although $\uHFKI$ itself does not seem to satisfy the strong functoriality properties that Heegaard Floer theory and its variants satisfy, we can still show that $\bar{\upsilon}$ and $\ubar{\upsilon}$ behave nicely under nonorientable cobordisms. In particular, we have the following results, which can be seen as analogues of the oriented involutive genus bound given by Juh\'{a}sz and Zemke in \cite{juhasz2020concordance}.

\begin{thm}
\label{mobiusthm}
Let $K$ be a knot in $S^3$ which bounds a smooth M{\" o}bius band in $B^4$. Then 
$\bar{\upsilon}(K)-\ubar{\upsilon}(K)\le 1$.
\end{thm}

\begin{thm}
\label{highergenusthm}
For any knot $K$ in $S^3$, we have 
\[
\begin{split}
    \bar{\upsilon}(K) - \upsilon(K) &\le \gamma_4(K)+1,\\
    \upsilon(K) - \ubar{\upsilon}(K) &\le \gamma_4(K)+1,\\
    \bar{\upsilon}(K) - \ubar{\upsilon}(K) &\le \gamma_4(K)+2.
\end{split}
\]
 Here, $\upsilon(K)$ is the upsilon invariant defined in \cite{ozsvath2017unoriented}.
\end{thm}

In \Cref{sec:involutive}, we will prove \Cref{mobiusthm} and \Cref{highergenusthm} and use them to derive \Cref{thm:lowerbound}.

\subsection{Questions} Recall that a torus knot of the form $T_{4,q}$ bounds a  M\"{o}bius band when $q=1,3,9,11$. \Cref{thm:4q} eliminates the cases $q \equiv 5,7 \pmod 8$ in the smooth category and tells us that that the set of all $q$ such that $T_{4,q}$ bounds a locally flat M\"{o}bius band has zero density, but there can still be infinitely many $q$ such that $T_{4,q}$ bounds a M\"{o}bius band in either smooth or topological category (see \Cref{rem:primefactors}). Thus the following question naturally arises.

\begin{qn} Does there exist an integer $q\not\in\{1,3,9,11\}$ such that $\gamma_4(T_{4,q})=1$? If such an integer $q$ exists, then are there infinitely many such integers $q$?\end{qn}

Moreover, all known counterexamples to Batson's conjecture are of the form $T_{p,q}$ where $p$ is even. No example with $p$ and $q$ both odd is currently known. We thus ask the following question.

\begin{qn} Do there exist relatively prime odd integers $p$ and $q$ such that $\gamma_4(T_{p,q})<\vartheta(T_{p,q})$?\end{qn}

In the topological category, \Cref{asymptotic} tells us that for any even $p$ with $\frac{p}{2}$ not a perfect square, $T_{p,q}$ does not bound a locally flat M\"{o}bius band for almost all $q$. However, this obstruction does not imply that the set of $q$ such that $T_{p,q}$ bounds a locally flat M\"{o}bius band is finite. In addition, the case when $\frac{p}{2}$ is a perfect square remains mysterious. Hence one can ask the following question; note that, when $p$ is odd, $\gamma_4(T_{p,q})=\vartheta(T_{p,q})=1$ whenever $q \equiv \pm 2 \pmod p$ (Proposition 1.6 in \cite{jabuka2019nonorientable}).

\begin{qn} Is there an even integer $p\ge 4$ such that $T_{p,q}$ bounds a locally flat M\"{o}bius band in $B^4$ for infinitely many integers $q$?\end{qn}

\subsection*{Organization.} This article is organized as follows. 
In \Cref{sec:background}, we recall some previously known results regarding lower and upper bounds on the nonorientable $4$-ball genus before proving \Cref{prop:longoextension} and Theorems \ref{thm:evencalculations} and \ref{thm:4q} in \Cref{sec:calculations}.
In \Cref{sec:involutive}, we develop a theory of involutive unoriented knot Floer homology and prove Theorems \ref{mobiusthm}, \ref{highergenusthm}, and \ref{thm:lowerbound}.
Lastly, in \Cref{sec:asymp}, we compute the linking form of $\Sigma(2,p,q)$, which is the branched double cover of $S^3$ along $T_{p,q}$, when $p$ is even, and use it to prove \Cref{asymptotic}.

\subsection*{Acknowledgements.} This work is the product of a research group formed under the auspices of the American Institute for Mathematics (AIM) in their virtual research community on 4–dimensional topology. We are grateful to AIM, and especially to the program organizers Miriam Kuzbary, Maggie Miller, Juanita Pinz\'{o}n-Caicedo, and Hannah Schwartz. 
The authors thank Peter Feller, Marco Golla, Stanislav Jabuka 
and the anonymous referees
for their useful comments on a draft of this paper.
The second author thanks Seungwon Kim for helpful conversations and acknowledges support by the
Institute for Basic Science (IBS-R003-D1). The fourth author was supported by the Swiss National Science Foundation Grant 181199.

\section{Background on the nonorientable $4$-ball genus}\label{sec:background}

In this section, we give an overview of the lower and upper bounds on the nonorientable $4$-ball genus provided in the literature, including a review of pinch moves, the pinch number, and related computations. 

\subsection{Obstructions for the existence of smooth nonorientable surfaces}

Even before the notion of nonorientable $4$-ball genus was born, Viro \cite{viro} proved that the figure-eight knot cannot bound a smoothly embedded M{\" o}bius band in $B^4$ using Witt classes of intersection forms of branched covers of $B^4$ branched over nonorientable surfaces. 

In \cite{Yasuhara}, Yasuhara formulated an obstruction using the classical signature $\sigma(K)$ and the Arf invariant $\arf(K)$ of a knot $K$ which can be stated as follows. 
\begin{lem}[{\cite[Proposition 5.1]{Yasuhara}}]\label{lem:yasuhara}
Let $K$ be a knot in $S^3$ that bounds a smoothly embedded M{\" o}bius band in $B^4$. Then
$
\sigma(K) + 4 \arf(K) \equiv 0 \text{ or } \pm 2 \pmod 8.
$
\end{lem}

\begin{rem}\label{rem:arf}
The Arf invariant 
satisfies \cite{levine}
\begin{align*}
\arf(K) = 
        \begin{cases} 0 &\mbox{if } \Delta_K(-1) \equiv \pm 1 \pmod{8},\\
1& \mbox{if } \Delta_K(-1) \equiv \pm 3 \pmod{8}. \end{cases}  
\end{align*}
For torus knots $T_{p,q}$ with relatively prime positive integers $p$ and $q$, it 
is given as follows. Without loss of generality, we assume that $q$ is odd. The determinant of $T_{p,q}$ is then given as
\begin{align*}
\det \left(T_{p,q}\right)&=\left \vert \Delta_{T_{p,q}}(-1)\right \vert= \begin{cases} 
    1 &\mbox{if } p\text{ is odd},\\
    q &\mbox{if } p \text{ is even}, 
\end{cases} \qquad \text{so}\\
   \arf \left(T_{p,q}\right)&= \begin{cases} 
    0 &\mbox{if } p \text{ is odd or } q\equiv \pm 1\pmod 8,\\
    1 &\mbox{if }  p \text{ is even and } q\equiv \pm 3 \pmod 8.
\end{cases}
\end{align*}
\end{rem}

Along the same lines as \Cref{lem:yasuhara}, Gilmer and Livingston showed the following.

\begin{lem}[{\cite[Theorem 10]{gilmerlivingston}}]
Let $K$ be a knot in $S^3$ that bounds a smoothly embedded punctured Klein bottle $F$ in $B^4$. If $\Sigma_2(B^4,F)$, the \emph{double branched cover of $B^4$ branched over $F$}, has a positive definite intersection form, then
$
\sigma(K) + 4 \arf(K) \equiv 0, 2 \text{ or } 4 \pmod 8.
$
Moreover, if $\Sigma_2(B^4,F)$ is negative definite, then $
\sigma(K) + 4 \arf(K) \equiv 0, 4 \text{ or } 6 \pmod 8$.
\end{lem}

It was an open question until work of Batson \cite{batson2012nonorientable} whether $\gamma_4(K) \leq 3$ for all knots $K$ in $S^3$.
Batson found a lower bound on $\gamma_4$ using Heegaard Floer homology and used it to show that $\gamma_4\left(T_{2k,2k-1}\right) = k-1$ for all $k \geq 1$. 
His lower bound was the following. 

\begin{lem}[\cite{batson2012nonorientable}]\label{lem:batson}
For any knot $K$ in $S^3$, we have $\frac{\sigma(K)}{2}-d(S_{-1}^3(K)) \leq \gamma_4(K)$, where 
$d(S_{-1}^3(K))$ denotes the correction term by  Ozsv\'ath and Szab\'o \cite{ozsvath2003absolutely} of the $(-1)$-surgery of $S^3$ along $K$, equipped with the unique $\text{Spin}^c$ structure.
\end{lem}

Later, Ozsv\'ath, Stipsicz and Szab\'o defined an additive concordance invariant $\upsilon(K)$ for any knot $K$ in $S^3$ using the Heegaard Floer knot complex \cite{ozsvath2017unoriented} and proved the following bound.

\begin{lem}[{\cite[Theorem 1.2]{ozsvath2017unoriented}}]\label{lem:OSSbound}
For any knot $K$ in $S^3$, we have $\left \vert \upsilon(K) - \frac{\sigma(K)}{2} \right \vert \leq \gamma_4(K)$.
\end{lem}

Using this lower bound, they reproved Batson's result that the nonorientable $4$-ball genus of a knot can be arbitrarily large by showing that $\gamma_4(\#_n T_{3,4} ) = n$ for all $n \geq 1$ \cite[Corollary 1.4]{ozsvath2017unoriented}. \\

There are now several lower bounds on $\gamma_4$ that look similar to that given in \Cref{lem:OSSbound}. First, Golla and Marengon \cite{gollamarengon} showed that $\gamma_4(K) \geq \frac{\sigma(K)}{2}-\min_{m \geq 0} \left\{m + 2V_m (\bar{K})\right\}$ for any knot $K$ in $S^3$, where $\bar{K}$ denotes the mirror of $K$ and $\{V_i\}_i$ are the concordance invariants explored in \cite{niwu} (and originally defined in \cite{rasmussen2004lens}). This bound agrees with the bound in \Cref{lem:OSSbound} for alternating knots and $L$-space knots, in particular for torus knots. In addition,
Daemi and Scaduto~\cite{daemi2020chern} and Ballinger~\cite{ballinger2020concordance} defined invariants $h_s$ and $t$ in gauge theory and Khovanov homology, respectively, from which similar inequalities can be obtained.
Again, for many knots, including alternating knots and torus knots, these bounds are not better than the one from \Cref{lem:OSSbound}.
Shortly after, Allen showed the following.

\begin{lem}[{\cite[Proposition 6.5]{allen2020nonorientable}}]\label{lem:allen}
Let $K$ be a knot in $S^3$ with $\sigma(K) < 2\upsilon(K)$ such that the \emph{double branched cover of $S^3$ along $K$}, denoted by $\Sigma_2(S^3,K)$, is an integer homology $3$-sphere and $\delta(K) < 0$, where $\delta(K)$ is a concordance invariant defined by Manolescu and Owens \cite{manolescu2007concordance}. Then $\upsilon(K) - \frac{\sigma(K)}{2} +1 \leq \gamma_4(K)$.
\end{lem}

It is well-known that the double cover of $S^3$ branched along the torus knot $T_{p,q}$ is the Brieskorn sphere $\Sigma(2,p,q)$, which is a homology sphere if and only if $p$ and $q$ are both odd. Therefore, Allen's lower bound does not apply to the infinite families of torus knots in Theorems \ref{thm:evencalculations} and \ref{thm:4q}.

Finally, given a knot $K$, one can obstruct the existence of a M{\"o}bius band bounded by $K$ (or possibly nonorientable surfaces with higher first Betti number) by computing Fr\o yshov's $h-$invariant of $\Sigma_2(S^3,K)$ and appealing to \cite[ Theorem 3]{froyshov} or by using Donaldson's Diagonalization Theorem as in \cite{jabukakelly}. Because the latter method has been used recently to help calculate the nonorientable 4-ball genus of all knots with 8 or 9 crossings in \cite{jabukakelly} and of all knots with 10 crossings in \cite{ghanbarian2020}, we will highlight this strategy.

The idea is as follows. Suppose $K$ bounds a M{\"o}bius band $M$ and that the double cover of $B^4$ branched along $M$ --- $\Sigma_2(B^4,M)$ --- is positive (resp. negative) definite. Suppose moreover that the double cover of $S^3$ branched along $K$ --- $\Sigma_2(S^3,K)$--- bounds a negative (resp. positive) definite 4-manifold $X$. Then the closed 4-manifold $X\cup_{\Sigma_2(S^3,K)}(-\Sigma_2(B^4,M))$ is negative (resp. positive) definite. By Donaldson's Diagonalization Theorem \cite{donaldson}, there exists a lattice embedding $(H_2(X),Q_X)\to\left(\mathbb{Z}^{\text{rank}(H_2(X))+1},-I\right)$. If one can show that there is no such embedding, then it follows that $K$ does not bound a M{\"o}bius band. See \cite{jabukakelly} for details. The downside to this obstruction is that it relies on finding a suitable definite 4-manifold bounded by $\Sigma_2(S^3,K)$. The more complicated the knot in question, the more nontrivial it can be to find such a definite 4-manifold and, indeed, such definite 4-manifolds need not exist.

\subsection{Obstructions for the existence of topologically locally flat nonorientable surfaces}

Given a knot $K$ in $S^3$, recall that $\gammatop(K)$ is the minimal first Betti number of locally flat nonorientable surfaces in $B^4$ bounded by $K$. Murakami and Yasuhara \cite{murakami2000four} proved the following lower bound in this category.

\begin{lem}[{\cite[Theorem 2.5]{murakami2000four}}]\label{lem:linkingpairing}
Suppose that a knot $K$ in $S^3$ bounds a locally flat null-homologous nonorientable surface of first Betti number $g$ in a compact $4$-manifold $W$ bounding $S^3$ with $H_1(W;\mathbb{Z})=0$. Then the linking form $\lambda$ on $H_1 (\Sigma_2(S^3,K);\mathbb{Z})$ splits into a direct sum $(G_1 ,\lambda_1) \oplus (G_2 ,\lambda_2)$, where the following conditions are satisfied:
\begin{itemize}
    \item there exists a square matrix $V$ with integer entries, whose size is $2b_2(W)+g$ and whose determinant is $\pm \vert G_1 \vert$, such that $\lambda_1$ is presented by $-V^{-1}$.
    \item $\lambda_2$ is metabolic.
\end{itemize}
\end{lem}

Using this lower bound, Gilmer and Livingston  \cite{gilmerlivingston} were able to prove that there exist knots with $\gammatop \ge 3$. Feller and Golla \cite[Proposition 1.4]{feller2020nonorientable} used it to show that if $p > 0$ with $p \equiv 5 \pmod 8$, then $T_{p,p\pm 1}$ does not bound a locally flat M{\" o}bius band in $B^4$. 
Moreover, by \cite[Proposition 5.6]{feller2020nonorientable}, for each odd prime $p$ and each choice of sign, there are infinitely many positive integers $k$ such that the knot $T_{2p, 2kp\pm 1}$ does not bound a locally flat M{\" o}bius band in $B^4$. To the best of the authors' knowledge, it is unknown whether $\gammatop$ can be arbitrarily large.

\subsection{Pinch moves and the pinch number}\label{sec:upperBounds}
Consider the torus as a square with opposite edges identified. Then the torus knot $T_{p,q}$ can be represented by the collection of parallel strands in the square with slope $\frac{p}{q}$. A \emph{pinch move} is a nonorientable band move obtained by attaching an untwisted band between two adjacent strands of $T_{p,q}$. The resulting knot still lies in the torus and is thus a torus knot, which in fact 
does not depend on the adjacent strands to which we choose to attach the band. A formula for the resulting torus knot was first provided in \cite{batson2012nonorientable} and proved in \cite{jabuka2019nonorientable}.

\begin{lem}[\cite{batson2012nonorientable},{\cite[Lemma 2.1]{jabuka2019nonorientable}}]
Let $p$ and $q$ be relatively prime positive integers and consider a diagram of the torus knot $T_{p,q}$ on the flat torus. After applying a pinch move to $T_{p,q}$, the resulting torus knot (up to orientation) is $T(|p-2t|,|q-2h|)$, where $t$ and $h$ are the integers uniquely determined by the
requirements \label{lem:pinch}
\begin{align*}
    &t\equiv -q^{-1}\pmod p \text{ and } t\in\{0,\ldots,p-1\},\\
    &h\equiv p^{-1}\pmod q \text{ and } h\in\{0,\ldots,q-1\}.
\end{align*}
\end{lem}

Suppose performing a pinch move on $T_{p,q}$ yields $T_{r,s}$. This pinch move can either be \emph{positive} or \emph{negative}; it is called positive if $rq-sp>0$ and negative if $rq-sp<0$ (c.f. \cite[Lemma 2.4]{jabuka2019nonorientable}). Note that since $T_{p,q}=T_{q,p}$ and $T_{r,s}=T_{s,r}$, performing a pinch move on $T_{q,p}$ yields $T_{s,r}$. Further note that the sign of a pinch move depends on the order of $p$ and $q$; that is, the sign of the pinch move performed on $T_{p,q}$ yielding $T_{r,s}$ is opposite of the sign of the pinch move performed on $T_{q,p}$ yielding $T_{s,r}$. Finally, the \emph{pinch number} of $T_{p,q}$, which we denote by 
$\vartheta(T_{p,q})$, is the minimal number of pinch moves needed to turn $T_{p,q}$ into the unknot (which is equivalent to the first Betti number of the nonorientable surface $F_{p,q}$ constructed by Batson \cite{batson2012nonorientable} discussed in the introduction).
Using Lemma \ref{lem:pinch}, we can prove the following result, which we will use later.

\begin{lem}\label{pinchlemma} Let $p\geq 2 $ and $k\ge 1$.
Performing a pinch move on the torus knot $T_{p,kp\pm1}$ yields the torus knot $T_{p-2,k(p-2)\pm1}$\footnote{When $p=2$ and $k=1$, this forumla gives $T_{0,-1}$. However (as the proof shows), Lemma \ref{lem:pinch} yields $T_{0,1}$. Since $T_{0,1}$ and $T_{0,-1}$ are isotopic to the unknot, we find it convenient to eliminate the use of absolute values, writing $T_{p-2,k(p-2)\pm1}$ instead of $T_{|p-2|,|k(p-2)\pm1|}$}. Consequently,
\begin{align*}
\vartheta(T_{p,kp\pm1})= \begin{cases}
\frac{p-1}{2} &\mbox{if } p \mbox{ is odd,} \\
\frac{p}{2} &\mbox{if } p \mbox{ is even and } kp\pm1\neq p-1,\\
\frac{p-2}{2} &\mbox{if } p \mbox{ is even and } kp\pm1= p-1.
\end{cases}
\end{align*}
Moreover, the pinch move performed on $T_{p,kp+1}$ is negative, and the pinch move performed on $T_{p,kp-1}$ is positive if $p>2$ and negative if $p=2$. 
\label{lem:pinchmoves}
\end{lem}

\begin{proof}
We perform a pinch move on $T_{p,kp\pm1}$ with $p,k$ as in the statement of the lemma. Set $t=p-1$ and $h=k(p-1)+1$. Then 
\[
\begin{split}
    t(kp+1)=kp^2+p-kp-1 &\equiv -1\pmod{p}, \\
    hp=p(kp+1)-kp &\equiv 1\pmod{kp+1}.
\end{split}
\]
Now since $|p-2t|=p-2$ and $|kp+1-2h|=k(p-2)+1$, \Cref{lem:pinch} shows that performing a pinch move on $T_{p,kp+1}$ yields $T_{p-2,k(p-2)+1}$. Moreover, this pinch move is clearly negative.

Next let $t=1$ and $h=k$. Then 
\[
\begin{split}
    t(kp-1) &\equiv -1\pmod{p}, \\
    hp=kp &\equiv 1\pmod{kp-1}.
\end{split}
\]
Note that: $|p-2t|=p-2$ for all $p$; $|kp-1-2h|=k(p-2)-1$ for all $p\ge3$; and $|kp-1-2h|=1$ when $p=2$. Thus \Cref{lem:pinch} shows that performing a pinch move on $T_{p,kp-1}$ yields $T_{p-2,k(p-2)-1}$ when $p\ge3$ and $T_{0,1}$ when $p=2$. Moreover, this pinch move is clearly positive if $p>2$ and negative if $p=2$.

Now, if $p$ is odd, then by performing $\frac{p-1}{2}$ pinch moves to $T_{p,kp\pm1}$, we obtain the torus knot $T_{1,k\pm1}$, which bounds a disk; thus $\vartheta(T_{p,kp\pm1})=\frac{p-1}{2}$. If $p$ is even and $kp\pm1\neq p-1$, then by performing $\frac{p-2}{2}$ pinch moves to $T_{p,kp\pm1}$, we obtain the torus knot $T_{2,2k\pm1}$, which bounds a M{\"o}bius band; thus $\vartheta(T_{p,kp\pm1})=\frac{p}{2}$. If $p$ is even and $kp\pm1= p-1$, then by performing $\frac{p-2}{2}$ pinch moves to $T_{p,p-1}$, we obtain the torus knot $T_{2,1}$, which bounds a disk; thus $\vartheta(T_{p,p-2})=\frac{p-2}{2}$.
\end{proof}

It was conjectured in \cite{batson2012nonorientable} that $\gamma_4(T_{p,q})=\vartheta(T_{p,q})$. This conjecture is clearly true in some basic cases, as it is not hard to observe that $\gamma_4(T_{2,q})=1=\vartheta(T_{2,q})$ for all odd $q$ and $\gamma_4(T_{3,q})=1=\vartheta(T_{3,q})$ for all $q$ relatively prime to 3. Jabuka and van Cott \cite{jabuka2019nonorientable} studied Batson's conjecture and showed that it is true for several infinite families of torus knots, using pinch moves and the lower bound from \Cref{lem:OSSbound}. Their main result is the following.

\begin{thm}[{\cite[Theorem 1.10 and Corollary 1.11]{jabuka2019nonorientable}}]
Let $p, q > 1$ be relatively prime integers and suppose that $q$ is odd. Suppose further that there is a sequence of $n \geq 1$ pinch moves $T_{p_i,q_i} \to T_{p_{i-1},q_{i-1}}$ starting at $T_{p,q}=T_{p_n, q_n}$ and ending at the unknot $T_{p_0,1}$ with $p_0 \geq 0$ for integers $p_i, q_i \geq 1$, $i \in \{1, \dots, n\}$. For all $i \in \{1, \dots, n\}$, let $\varepsilon_i \in \{\pm 1\}$ be $+1$ if the pinch move $T_{p_i,q_i} \to T_{p_{i-1},q_{i-1}}$ is positive and $-1$ otherwise. For all $i \in \{1, \dots, n-1\}$, $m_i = \frac{p_{i+1}+\varepsilon_i \varepsilon_{i+1}p_{i-1}}{p_i}$ is an even integer.
\begin{enumerate}[(a)]
    \item Let $p$ be odd and $p > q$. 
    If $q_1 \equiv \varepsilon_1 \pmod{4}$ and $m_i \equiv 2 \pmod 4$ for all $i \in \{1, \dots, n-1\}$, then $\gamma_4\left(T_{p,q}\right) = n.$
    \item Let $p$ be even. If $\epsilon_i =1$ for all $i\in \{1, \dots,n\}$, then $\gamma_4\left(T_{p,q}\right) = n.$
\end{enumerate}
In both cases, $\gamma_4\left(T_{p,q}\right) = \vartheta\left(T_{p,q}\right) =n$.
\label{thm:jabukavanCottpinch}
\end{thm}

\Cref{thm:jabukavanCottpinch} heavily relies on the lower bound $|\upsilon - \frac{\sigma}{2}|$ on $\gamma_4$ provided by Ozsv\'ath, Stipsicz and Szab\'o (see \Cref{lem:OSSbound}). In fact, Jabuka and Van Cott give formulas for this lower bound for all torus knots $T_{p,q}$ in terms of the integers $p_i, q_i, \varepsilon_i$ and $m_i$ from \Cref{thm:jabukavanCottpinch}; see \cite[Theorems 1.7 and 1.8]{jabuka2019nonorientable}.
Using these formulas along with Lemma \ref{lem:pinchmoves} and Theorem \ref{thm:jabukavanCottpinch}, one can prove the following.

\begin{cor} \hfill
\begin{enumerate}[(a)]
    \item If $p$ is odd and $q\equiv p+1\text{ or }2p-1\pmod{2p}$, then $\gamma_4(T_{p,q})=\frac{p-1}{2}$.
    \item If $p$ is even and $q\equiv p-1 \text{ or }2p-1\pmod{2p}$, then $\gamma_4(T_{p,q})\in\left\{\frac{p-2}{2},\frac{p}{2}\right\}$.
\end{enumerate}
\label{cor:jabukavanCottpinch}
\end{cor}

As mentioned in the introduction, Lobb \cite{lobb} disproved Batson's conjecture by showing that $\gamma_4(T_{4,9})=1$, while $\vartheta(T_{4,9})=2$. This was achieved by finding a nonorientable band move from $T_{4,9}$ to the Stevedore knot, which is slice. Generalizing his approach, an infinite family of counterexamples was given by Longo in \cite{longo}.

\begin{thm}[\cite{longo}] For $n\ge2$, we have $\vartheta\left(T_{4n,(2n\pm 1)^2}\right)=2n$, while $\gamma_4\left(T_{4n,(2n\pm 1)^2}\right)\le 2n-1$. \label{thm:longo}
\end{thm}

In her masters thesis, Tairi \cite{tairi} gave another infinite family of counterexamples to Batson's conjecture.

\begin{thm}[\cite{tairi}]\hfill
\begin{enumerate}[(a)]
    \item $\vartheta(T_{4,11})=2$, while $\gamma_4(T_{4,11})=1$.
    \item Let $n\ge0$ and $m\ge 2$. Then $\vartheta(T_{4n+2m+2,10n+6m+5})=m+1$, while $\gamma_4(T_{4n+2m+2,10n+6m+5})=m$.
\end{enumerate}\label{thm:tairi}
\end{thm}

As in Lobb's proof that $T_{4,9}$ bounds a M{\" o}bius band, Tairi's proof that $T_{4,11}$ bounds a M{\" o}bius band relies on finding a nonorientable band move from $T_{4,11}$ to the Stevedore knot. To show that the torus knot $T_{4n+2m+2,10n+6m+5}$ bounds a nonorientable surface $F$ with $g(F)=m$, Tairi shows that performing $m-1$ pinch moves turns $T_{4n+2m+2,10n+6m+5}$ into $T_{4,11}$, which satisfies $\gamma_4(T_{4,11})=1$. 

Using this idea, one can start with any torus knot whose nonorientable 4-ball genus is strictly less than its pinch number and create an infinite family of such torus knots by reversing the pinch move operation. For each of Longo's torus knots in \Cref{thm:longo}, we can then produce an infinite family of such torus knots.

\begin{prop}
For all $n\ge2$ and $k\ge0$, we have\label{prop:families}
\begin{align*}
    \vartheta(T_{4n+2k,(4n+2k)(n\pm1)+1})=2n+k, \qquad \text{while } \gamma_4(T_{4n+2k,(4n+2k)(n\pm1)+1})\le 2n+k-1.
\end{align*}
\end{prop}

\begin{proof}
By repeatedly applying Lemma \ref{lem:pinchmoves}, performing $k$ pinch moves on $T_{4n+2k,(4n+2k)(n\pm1)+1}$ yields the torus knot $T_{4n,4n(n\pm1)+1}=T_{4n,(2n\pm1)^2}$. By \Cref{thm:longo}, $\vartheta(T_{4n,(2n\pm1)^2})=2n$ and $\gamma_4(T_{4n,(2n\pm1)^2})\le 2n-1$. Consequently, $\vartheta(T_{4n+2k,(4n+2k)(n\pm1)+1})=2n+k$ and $\gamma_4(T_{4n+2k,(4n+2k)(n\pm1)+1})\le 2n+k-1$.
\end{proof}

In \Cref{sec:calculations}, we will see that when $n$ is even, $\gamma_4(T_{4n+2k,(4n+2k)(n\pm1)+1})= 2n+k-1$. In particular, for $k=0$, this shows that half of Longo's examples in \Cref{thm:longo} satisfy $\gamma_4(T_{4n+2k,(4n+2k)(n\pm1)+1})=\vartheta(T_{4n+2k,(4n+2k)(n\pm1)+1})-1$.

\section{Calculations: Proofs of Proposition \ref{prop:longoextension}, Theorem \ref{thm:4q}, and Theorem \ref{thm:evencalculations}}\label{sec:calculations}

In this section, we prove \Cref{thm:4q}, \Cref{prop:longoextension}, and \Cref{thm:evencalculations}. In particular, we use Lemma \ref{lem:pinchmoves} to find upper bounds on $\gamma_4$ for these families and \Cref{thm:lowerbound} to find lower bounds.

\subsection{The family $T_{4,q}$}\label{subsec:4q} By Lemma \ref{lem:pinchmoves}, $\vartheta(T_{4,q})=2$ for all $q\ge 5$. Thus we have an obvious upper bound $\gamma_4(T_{4,q})\le2$ for all $q$. Now, from the formula
\begin{align*}
t^{3n}\Delta_{T_{4,2n+1}}(t) &= \frac{\left(t^{8n+4}-1\right)\left(t-1\right)}{\left(t^4-1\right)\left(t^{2n+1}-1\right)} \\
&= \frac{\left(t^{2n+1}+1\right)\left(t^{4n+2}+1\right)}{\left(t+1\right)\left(t^{2}+1\right)} = \left( \sum_{k=0}^{2n} (-t)^k \right)  \left( \sum_{k=0}^{2n} (-t^2)^k \right), 
\end{align*}
 we see that
\begin{align*}
\Delta_{T_{4,2n+1}}(t) &= \sum_{k=0}^{6n}a_k t^{k-3n}
\end{align*}
where
\begin{align*}
a_0 &= 
1, \qquad a_{6n} = 
1, \qquad a_{3n+1} = 0,\\
    a_{3n} &= 
    \begin{cases}
(-1)^{3n/2} \frac{1+(-1)^{n}}{2} 
&\mbox{if } n\text{ is even},\\
(-1)^{(3n+1)/2} \frac{1+(-1)^{(n-1)/2}}{2} 
&\mbox{if } n \text{ is odd}, \end{cases} \qquad \\
a_{3n+2} &=
    \begin{cases}
(-1)^{(3n+2)/2} \frac{1+(-1)^{n}}{2} 
&\mbox{if } n\text{ is even},\\
(-1)^{(3n+3)/2} \frac{1+(-1)^{(n-1)/2}}{2}
&\mbox{if } n \text{ is odd}.\end{cases}
\end{align*}
Thus 
\begin{align*}
\Delta_{T_{4,2n+1}}(t) &= \begin{cases}
t^{-3n}-\cdots - 1 + t^2 + \cdots + t^{3n} &\mbox{if } n \equiv 2 \text{ or } 3 \pmod 4,\\
t^{-3n}-\cdots +1 - t^2 + \cdots + t^{3n} &\mbox{if } n \equiv 0 \text{ or } 1 \pmod 4. \end{cases}
\end{align*}

Hence, applying \Cref{thm:lowerbound} gives $\gamma_4(T_{4,q})=2$ for all $q\equiv 5,7\pmod8$.

On the other hand, Theorem \ref{thm:lowerbound} cannot be applied to $T_{4,q}$ when $q\equiv 1\text{ or }3\pmod 8$. However, the obstruction from the linking form on the first homology of the double cover of  $T_{4,q}$ from \Cref{lem:linkingpairing} does obstruct $T_{4,q}$ from bounding a M{\" o}bius band for infinitely many $q\equiv 1\text{ or }3\pmod8$. 
In fact, \Cref{asymptotic} (proved in \Cref{sec:involutive}) implies that 
\[
\lim_{n\rightarrow\infty} \frac{\sum_{k=1}^{n} \gammatop (T_{4,8k+1})}{n} = \lim_{n\rightarrow\infty} \frac{\sum_{k=1}^{n} \gammatop (T_{4,8k+3})}{n} =2.
\]
This means that Batson's conjecture is also true for torus knots of the form $T_{4,8k+1}$ and $T_{4,8k+3}$ for almost all $k \geq 0$. As mentioned in Remark \ref{rem:primefactors}, we can explicitly find values of $q\equiv1\text{ or }3\pmod8$ such that $T_{4,q}$ does not bound a locally flat M{\"o}bius band. See Example \ref{exmp:4q} for details.

It is worth noting here that \Cref{thm:lowerbound} gives a sharper lower bound than the other known lower bounds detailed in Section \ref{sec:background} when $q\equiv 5\text{ or }7\pmod8$. 
Using the recursive formulas for $\sigma(T_{p,q})$ \cite{gordonlitherlandmurasugi}, $\upsilon(T_{p,q})$ \cite{fellerkrcatovich}, and $d(S^3_{-1}(T_{-p,q}))$ \cite{batson2012nonorientable}, and the formula in Remark \ref{rem:arf}, we see that 

\begin{align*}
\sigma(T_{4,q})&=
\begin{cases}
-2(q-1) & \text{ if } q\equiv 1\pmod 4,\\
-2q & \text{ if } q\equiv 3\pmod 4,
\end{cases}\\
\upsilon(T(4,q))&=-(q-1)\text{ for all odd }q,\\
d(S^3_{-1}(T_{-4,q}))&=
\begin{cases}
q-1 & \text{ if } q\equiv 1\pmod 8,\\
q-1 & \text{ if } q\equiv 3\pmod 8,\\
q+1 & \text{ if } q\equiv 5\pmod 8,\\
q+1 & \text{ if } q\equiv 7\pmod 8,
\end{cases}\\
\arf(T_{4,q})&=
\begin{cases}
0 & \text{ if } q\equiv 1,7\pmod 8,\\
1 & \text{ if } q\equiv 3,5\pmod 8.
\end{cases}
\end{align*}

It follows that

\begin{align*}
(\sigma+4\arf)(T_{4,q})\pmod 8&=
\begin{cases}
0 & \text{ if } q\equiv 1\pmod 8,\\
2 & \text{ if } q\equiv 3\pmod 8,\\
4 & \text{ if } q\equiv 5\pmod 8,\\
6 & \text{ if } q\equiv 7\pmod 8,\\
\end{cases} \\
    \left \vert \left(\upsilon - \frac{\sigma}{2}\right)\left(T(4,q)\right)\right \vert &= \begin{cases}
0 &\mbox{if } q\equiv 1\pmod4,\\
1 &\mbox{if } q\equiv 3\pmod4, \end{cases}\\
\frac{\sigma(T_{-4,q})}{2} - d(S^3_{-1}(T_{-4,q}))&=
\begin{cases}
0 & \text{ if } q\equiv 1\pmod 8,\\
1 & \text{ if } q\equiv 3\pmod 8,\\
-2 & \text{ if } q\equiv 5\pmod 8,\\
-1 & \text{ if } q\equiv 7\pmod 8.\\
\end{cases}
\end{align*}

We can see that \Cref{lem:yasuhara} only applies when $q\equiv 5\pmod 8$ and 
the lower bound $|\frac{\sigma}{2}-\upsilon|$ from \Cref{lem:OSSbound} does not give as sharp a lower bound as \Cref{thm:lowerbound} when $q\equiv 5,7\pmod 8$. Consequently, the lower bounds due to Golla and Marengon, Ballinger, and Daemi and Scaduto mentioned in Section \ref{sec:background} are not as sharp either. 
Notice that Batson's lower bound (from \Cref{lem:batson}) shows that $\frac{\sigma(T_{-p,q})}{2}-d(S^3_{-1}(T_{-p,q}))\le \gamma_4(T_{-p,q})=\gamma_4(T_{p,q})$ (since $T_{-p,q}$ bounds a nonorientable surface of first Betti number $g$ if and only if $T_{p,q}$ bounds a nonorientable surface of first Betti number $g$). Therefore, Batson's lower bound is also not as sharp as \Cref{thm:lowerbound} when $q\equiv 5,7\pmod8$. 
Meanwhile, Allen's bound from \Cref{lem:allen} does not apply since the double cover of $S^3$ branched over $T_{4,q}$ is not an integral homology sphere. Finally, the obstruction from \Cref{lem:linkingpairing} applies to the cases $q\equiv 5\text{ or }7\pmod8$ in the same way as it applies to the cases $q\equiv 1\text{ or }3\pmod8$; however, it does not obstruct every such torus knot from bounding a M{\"o}bius band (see Example \ref{exmp:4q} for details). 

We finally mention that the techniques of \cite{jabukakelly} can be used to show that $T_{4,q}$ does not bound a M{\"o}bius band for all $q\equiv 5\text{ or }7\pmod8$, although it requires considerable more work than simply applying Theorem \ref{thm:lowerbound}. To use this obstruction, one must first locate certain definite 4-manifolds bounded by the double branched cover along $T_{4,q}$ and then obstruct the existence of certain lattice embeddings given by Donaldson's Diagonalization Theorem \cite{donaldson}. 

\subsection{A Detour: Semigroups and the Alexander polynomial} To apply Theorem \ref{thm:lowerbound}, one must compute the Alexander polynomial. Such computations are often tedious, especially when considering infinite families of knots. 
To help with these calculations we exploit a relationship between the Alexander polynomial of torus knots (or in general, algebraic knots) and certain sub-semigroups of $\mathbb{Z}_{\ge 0}$.

Recall that a knot $K$ is an \emph{$L$-space knot} if it admits a positive Dehn surgery which is an $L$-space; see \cite[Definition 1.1]{ozsvath2005knot} for the definition of an $L$-space. 
It is proven in \cite[Theorem 1.2]{ozsvath2005knot} that when $K$ is an $L$-space knot, the Alexander polynomial of $K$ takes the form 
\begin{equation}
\Delta_K(t)=(-1)^m + \sum_{i=1}^{m} (-1)^{m-i}\left(t^{n_i}+t^{-n_i}\right)
\label{eqn:alex}
\end{equation}
for some sequence of positive integers $0<n_1<\cdots <n_m =g(K)$. 

The \emph{formal semigroup} of an $L$-space knot $K$, defined in \cite[Section 2.1]{wang2018semigroups} (see also \cite{bodnarceloriagolla}), is the subset $S_K \subset \mathbb{Z}_{\ge 0}$ satisfying the following equation in $\mathbb{Z}[[t,t^{-1}]]$:
\begin{equation}
t^{\deg \Delta_{K}} \cdot \Delta_{K} (t)= (1-t)\cdot \left(\sum_{i\in S_K} t^i \right).
\label{eqn:alexsemigroup}
\end{equation}

Recall that a knot $K$ is an \emph{algebraic knot} if it arises as a knot associated to a plane curve singularity in $\mathbb{C}^2$. It is known that algebraic knots are $L$-space knots \cite{hedden2009} and that the formal semigroup $S_K$ of an algebraic knot $K$ is an actual semigroup and it agrees with the traditionally defined semigroup for algebraic knots \cite{wall2004singular}. 
Torus knots are algebraic knots because they are precisely the knots associated to the singularity $z^p -w^q =0$ for $z,w \in \mathbb{C}$. 
Therefore, each torus knot gives rise to a semigroup. Moreover, it is well-known that the semigroup for $T_{p,q}$ is $S_{T_{p,q}}=p\mathbb{Z}_{\ge0}+q\mathbb{Z}_{\ge0}$ (see, e.g., \cite{borodziklivingston2014})
and that the Alexander polynomial of $T_{p,q}$ is the Laurent polynomial given by $$\displaystyle\Delta_{T_{p,q}}(t)=t^{-\frac{(p-1)(q-1)}{2}}\cdot \frac{(t-1)(t^{pq}-1)}{(t^p -1)(t^q -1)},$$ which has degree $\frac{(p-1)(q-1)}{2}$.

\begin{exmp}
Consider the torus knot $K=T_{3,5}$. Its Alexander polynomial is given by $$\Delta_K(x) \doteq x^{-4}\cdot\frac{(x-1)(x^{15}-1)}{(x^3-1)(x^5-1)}=x^{-4}-x^{-3}+x^{-1}-1+x-x^3+x^4.$$ On the other hand, the semigroup generated by $3$ and $5$ is given by $$\langle 3,5\rangle = \{ 0,3,5,6 \}\cup \mathbb{Z}_{\ge 8}.$$ This matches the equality given by \Cref{eqn:alexsemigroup}:
\[
\begin{split}
    x^4(x^{-4}-x^{-3}+x^{-1}-1+x-x^3+x^4) &= 1-x+x^3-x^4+x^5-x^7+x^8 \\
    &= (1-x)\left(1+x^3+x^5+x^6+\sum_{i=8}^{\infty}x^i\right).
\end{split}
\]
\end{exmp}

Since algebraic knots are $L$-space knots, their Alexander polynomials are as in \Cref{eqn:alex}. Coupling this with \Cref{eqn:alexsemigroup}, we have the following lemma, which was also stated (and proven) in a slightly different way in \cite{borodzik2019involutive}.

\begin{lem}[{\cite[Proposition 3.8]{borodzik2019involutive}}]
\label{semigrouplemma}
Let $K$ be an algebraic knot whose Alexander polynomial has leading terms $t^d$ and $t^{-d}$. Then, for any integer $i$, the coefficient of the term $t^i$ of $\Delta_K(t)$ is $1$ if $d+i-1\notin S_K$ and $d+i\in S_K$, $-1$ if $d+i-1\in S_K$ and $d+i\notin S_K$, and $0$ if $d+i-1$ and $d+i$ are either both contained in $S_K$ or both not contained in $S_K$.
\end{lem}

\subsection{The family $T_{p,(2k+1)p-1}$, where $p$ is odd}\label{subsec:podd} We now prove the first statement of \Cref{thm:evencalculations}\ref{item:p-1} (the second statement of \Cref{thm:evencalculations}\ref{item:p-1} is due to Corollary \ref{cor:jabukavanCottpinch}). Suppose $p$ is odd and let $q\equiv p-1\pmod{2p}$. Set $q=2kp+(p-1)=(2k+1)p-1$. Notice that by Lemma \ref{lem:pinchmoves}, $\vartheta(T_{p,(2k+1)p-1})=\frac{p-1}{2}$ and so $\gamma_4(T_{p,(2k+1)p-1})\le\frac{p-1}{2}$.

Consider the semigroup $S_K$ for the torus knot $K=T_{p,(2k+1)p-1}$, which is generated by $p$ and $(2k+1)p-1$. The degree of the Alexander polynomial $\Delta_K$ of $K$ is $d=\frac{(p-1)((2k+1)p-2)}{2}.$ We first claim that $d+i\notin S_K$ for all $0\le i\le \frac{p-3}{2}$. Suppose otherwise. Then there exist non-negative integers $a$ and $b$ satisfying 
$$ap+b((2k+1)p-1)=\frac{(p-1)((2k+1)p-2)}{2}+i.$$
\noindent Solving for $a$, we find
$$a=\frac{(p-1)(2k+1)p-2p-2b(2k+1)p}{2p}+\frac{1+i+b}{p}.$$
Since $p$ is odd and $a\in\mathbb{Z},$ we necessarily have that $1+i+b=\alpha p$ for some integer $\alpha$. Since $b\ge0$, we have that $b=\alpha p-i-1\ge0$, implying that $\alpha\ge1$. 
Therefore


\begin{align*}
    a &= \frac{(p-1)(2k+1)p-2p-2(\alpha p -i -1)(2k+1)p+2\alpha p}{2p}
    \\&= \frac{-2\alpha p((2k+1)p-1)
    +(2k+1)p(p-1+2(i+1))
    -2p}{2p}
    \leq \frac{(2k+1)p(-2p+2i+p+1)
    }{2p}
    \\&\leq \frac{-2p(2k+1)}{2p}<0
\end{align*}
using $i\le \frac{p-3}{2}$ in the second-to-last inequality, which contradicts 
the assumption that $a\ge0$. Thus $d+i\notin S_K$ for all $0\le i\le \frac{p-3}{2}$. It follows from \Cref{semigrouplemma} that since the constant term of the Alexander polynomial is nonzero and $d\notin S_K$, the constant term of the Alexander polynomial is $-1$. Moreover, since $d+i\notin S_K$ for all $0\le i\le \frac{p-3}{2}$, we can further deduce by \Cref{semigrouplemma} that the first nontrivial term of positive degree in $\Delta_K$ has exponent at least $\frac{p-1}{2}$. Finally, setting $a=0$ and $b=\frac{p-1}{2}$, we have that $ap+b((2k+1)p-1)=\frac{((2k+1)p-1)(p-1)}{2}=d+\frac{p-1}{2}.$ Therefore, $d+\frac{p-1}{2}\in S_K$. By \Cref{semigrouplemma}, the first nontrivial term of positive degree in $\Delta_K$ has exponent $\frac{p-1}{2}.$ By Theorem \ref{thm:lowerbound}, $\gamma_4(T_{p,(2k+1)p-1})\ge\frac{p-3}{2}.$ Thus we have shown $\gamma_4(T_{p,(2k+1)p-1})\in\left\{\frac{p-3}{2},\frac{p-1}{2}\right\}$.

\subsection{The family $T_{p,(2k+1)p+1}$, where $p$ is even}\label{subsec:peven} Finally, we prove the middle case of \Cref{thm:evencalculations}\ref{item:c}(the other cases of \Cref{thm:evencalculations}\ref{item:c} are due to Corollary \ref{cor:jabukavanCottpinch}). Let $q\equiv p+1\pmod{2p}$. Set $q=2kp+p+1=(2k+1)p+1$. By Lemma \ref{lem:pinchmoves}, $\vartheta(T_{p,(2k+1)p+1})=\frac{p}{2}$ and thus $\gamma_4(T_{p,(2k+1)p+1})\le \frac{p}{2}$. 

As above, let $S_K$ denote the semigroup for $T_{p,(2k+1)p+1}$, which is generated by $p$ and $(2k+1)p+1$. We claim that $d+i\notin S_K$ for all $0\le i \le \frac{p-2}{2}$, where $d=\frac{(p-1)(2k+1)p}{2}$ is the degree of the Alexander polynomial. Indeed, if there exist nonnegative integers $a$ and $b$ satisfying 
$$ap+b((2k+1)p+1)=d+i$$
then 
$$a=\frac{(p-1)(2k+1)}{2}-b(2k+1)+\frac{i-b}{p}.$$
Since $p$ is even and $a\in\mathbb{Z}$, $\frac{i-b}{p/2}$ must be odd. Thus $i-b=\alpha \frac{p}{2},$ where $\alpha$ is an odd integer. Consequently, $0\le b=i-\alpha \frac{p}{2}\le -\frac{\alpha -1}{2}p-1$ (using $i \le \frac{p-2}{2}$) and so $\alpha\le -1$ (using that $\alpha$ is odd). Thus we have 
$$a=\frac{(p-1)(2k+1)}{2}-b(2k+1)+\frac{i-b}{p}
= \frac{2k+1}{2}\underbrace{\left(p-1-2\left(i-\alpha \frac{p}{2}\right)\right)}_{=p(\alpha+1)-(2i+1)}+\frac{\alpha}{2}
\le -\frac{(2i+1)(2k+1)}{2}<0,$$
contradicting the assumption that $a\ge 0.$ Moreover, setting $a=\frac{(p-1)(2k+1)+1}{2}$ and $b=0$, we see that $ap+b((2k+1)p+1)=d+\frac{p}{2}\in S_K.$ Now arguing as in \Cref{subsec:podd}, we have that 
the Alexander polynomial of $K$ has constant term $-1$ and the first nontrivial term of positive degree has exponent $\frac{p}{2}$. Thus by Theorem \ref{thm:lowerbound}, $\gamma_4(T_{p,(2k+1)p-1})\ge \frac{p-2}{2}$.

\subsection{The family $T_{4n+2k,(4n+2k)(n\pm1)+1}$}
Finally, we finish the proof of \Cref{prop:longoextension}. Notice that since $n$ is even, the torus knots in this family are contained in the set of torus knots considered in \Cref{subsec:peven} (c.f.~\Cref{rem:specialcase}). Thus we have $\gamma_4(T_{4n+2k,(4n+2k)(n\pm1)+1})\in\{2n+k-1,2n+k\}$. Now, by \Cref{prop:families}, we have $\gamma_4(T_{4n+2k,(4n+2k)(n\pm1)+1})=2n+k-1$.

\begin{rem} 
By \Cref{lem:pinchmoves}, all of the pinch moves used for the $T_{p,(2k+1)p+1}$ family in \Cref{subsec:peven} are negative; since $p$ is even, Theorem 1.8 in \cite{jabuka2019nonorientable} tells us that the lower bound given by Lemma \ref{lem:OSSbound} is 0. 
Similarly, by \Cref{lem:pinchmoves}, all of the pinch moves used for the $T_{p,(2k+1)p-1}$ family in \Cref{subsec:podd} are positive. However, since $(2k+1)p-1$ is even, to apply Theorem 1.8 in \cite{jabuka2019nonorientable}, we reverse the role of $p$ and $(2k+1)p-1$ so that all pinch moves are negative. By doing so we find that the lower bound given by Lemma \ref{lem:OSSbound} is 0.
Finally, one can see that the lower bound in \Cref{lem:batson} is, in general, not as sharp as \Cref{thm:lowerbound} for these families (e.g. for $T_{8,9}$, the lower bound given by \Cref{lem:batson} is $-4$). Thus Theorem \ref{thm:lowerbound} gives a much sharper lower bound than the existing lower bounds discussed in \Cref{sec:background}.
\end{rem}

\section{Involutive obstruction and genus bounds from the unoriented link Floer TQFT}\label{sec:involutive}
In this section, we will prove Theorems \ref{mobiusthm}, \ref{highergenusthm}, and \ref{thm:lowerbound}.
We begin by briefly recalling Fan's disoriented link Floer TQFT, as defined in \cite{fan2019unoriented}. Although Fan's TQFT accounts for unorientable cobordisms between links in $S^3$, there are a number of restrictions of which one should be aware.
\begin{defn}
A \emph{disoriented link} (in $S^3$) is a quadruple $\mathcal{L}=(L,\p,\q,\bfl)$ where $L$ is a link in $S^3$, $\p$ and $\q$ are finite subsets of $L$ satisfying $\p\cap \q=\emptyset$, such that points in $\p$ and points in $\q$ appear alternately on each component of $L$, and $L\backslash (\p\cup\q)$ consists of $2n$ arcs $\bfl=\{l_1,\cdots,l_{2n}\}$ for some integer $n$.
\end{defn}
\begin{rem}
For simplicity, we will occasionally refer to a disoriented link $\mathcal{L}=(L,\p,\q,\bfl)$ and its underlying link $L$ interchangeably.
\end{rem}

Given a disoriented link $\mathcal{L}=(L,\p,\q,\bfl)$, where $\p=\{p_1,\dots,p_n\}$ and $\q=\{q_1,\cdots,q_n\}$, choose any orientation on $L$. Viewing the points $\p$ and $\q$ on $L$ as $z$-basepoints and $w$-basepoints respectively gives the link Floer chain complex $CFL(S^3,\mathcal{L})$ a structure of a curved complex of modules over the ring $\mathbb{F}_2[U_1,\cdots,U_n,V_1,\cdots,V_n]$, as shown in \cite{zemke2019link}. The curvature of this complex is given by the formula
\[
\partial_{CFL(S^3,\mathcal{L})}^2 = \sum_{K\in \pi_0(L)} U_{p_{K,1}}V_{q_{K,1}}+V_{q_{K,1}}U_{p_{K,2}}+\cdots+V_{q_{K,t}}U_{p_{K,1}},
\]
where the points of $\p$ and $\q$ on a component $K$ are labeled as $p_{K,1},q_{K,1},\cdots,p_{K,t},q_{K,t}$ as one travels along $L$ with the chosen orientation. Consider the truncation of $CFL(S^3,\mathcal{L})$ by identifying all of the variables $U_i$ and $V_i$ with a single formal variable $U$, and denote the resulting complex by $CFL^{\prime}(S^3,\mathcal{L})$. Then the curvature vanishes and thus $\uCFL (S^3,\mathcal{L})$ is a chain complex over the ring $\mathbb{F}_2[U]$.

Note that $CFL(S^3,\mathcal{L})$ is endowed with two gradings $\mathbf{gr}_\p$ and $\mathbf{gr}_\q$. The truncation is not compatible with this bigrading, and we only get a single grading, called the (normalized) $\delta$-grading, which is given by $\delta =\frac{1}{2}(\sigma(L)+gr_\p +gr_\q )$. Note that unlike in the setting of $CFL(S^3,\mathcal{L})$, $U$ lowers the normalized $\delta$ grading by $1$. The resulting $\delta$-graded chain complex $\uCFL (S^3,\mathcal{L})$ is called the \emph{unoriented link Floer chain complex} of $\mathcal{L}$ and its homology is called the \emph{unoriented link Floer homology}. When $\mathcal{L}$ is a knot, we write $\uCFK$ instead of $\uCFL$ and similarly $CFK$ for $CFL$. 

\begin{defn}
A \emph{disoriented link cobordism} (in $S^3 \times I$) between disoriented links $\mathcal{L}_1 =(L_1,\p_1,\q_1,\bfl_1)$ and $\mathcal{L}_2 =(L_2,\p_2,\q_2,\bfl_2)$ is a pair $\mathcal{F}=(F,a)$ where $F$ is a smoothly embedded surface in $S^3 \times I$ such that $\partial F= (L_1 \times \{0\}) \cup (L_2 \times \{1\})$ and $a$ is an oriented $1$-manifold which is properly embedded in $(F,\partial F)$ such that $F\backslash a$ is orientable and $\partial a=\q_1 -\p_1 +\p_2 -\q_2$.
\end{defn}

The main theorem of \cite{fan2019unoriented} is as follows. Note that, while the theorem is originally stated on the level of homology, one can observe directly from its proof that it actually holds on chain level up to (a possibly non-unique) chain homotopy. Slightly more specifically, the map $F_S$ is constructed by decomposing the disoriented link cobordism into a composition of four types of elementary disoriented link cobordisms~\cite[Chapter 2]{fan2019unoriented} to each component of which one assigns a $\delta$-grading preserving $\mathbb{F}_2[U]$-module map. Different decompositions into elementary cobordisms are related by a collection of moves (see \cite[Section 7.1]{fan2019unoriented}), each of which preserves the filtered chain homotopy class---see Step 6 in the proof of \cite[Theorem 1.3.1]{fan2019unoriented}.

\begin{thm}[{\cite[Theorem 1.3.1]{fan2019unoriented}}]
\label{fansthm}
Let $S$ be a disoriented link cobordism between (disoriented) links $L_1$ and $L_2$. Then there is a $\delta$-grading-preserving $\mathbb{F}_2[U]$-module map 
\[
F_S \colon \uCFL (S^3,L_1) \rightarrow \uCFL (S^3,L_2),
\]
called the cobordism map induced by $S$, whose chain homotopy class depends only on the isotopy class of $S$, as a disoriented link cobordism. This map is functorial in the following sense: if $S_1$ and $S_2$ are disoriented link cobordisms between $L_1$ and $L_2$ and between $L_2$ and $L_3$, respectively, then we have functoriality up to chain homotopy: 
\[
F_{S_2 \circ S_1} \sim F_{S_2} \circ F_{S_1}.
\]
\end{thm}

\begin{rem}
\label{oriented-equivalent}
For oriented link cobordisms, the notion of disoriented cobordism is equivalent to the notion of decorated cobordism in \cite{juhasz2016cobordisms} and \cite{zemke2019link}. This equivalence is a little subtle. The decoration in a decorated cobordism with underlying surface $\Sigma$ is an embedded $1$-manifold in $\Sigma$ whose boundary is required to be disjoint from the basepoints on the link. On the other hand, we have that the decoration in a disoriented cobordism is required to coincide with the basepoints on the link. Nevertheless, there is a notion of compatibility between disoriented and decorated orientable link cobordisms~\cite[Remark 4.14]{gong2021non} and it is known that the cobordism maps from Zemke's oriented link Floer TQFT induce the cobordism maps from Fan's unoriented link Floer TQFT under truncation for appropriate compatible cobordisms~\cite[Remark 4.15]{gong2021non}.
\end{rem}

For later use, we will give some trivial examples of disoriented cobordisms for which Fan's cobordism maps can be easily computed.
\begin{exmp}
\label{unknottedmobius}
Consider the unknot $U$, made into a disoriented knot $\mathcal{U}$ by choosing one $\p$ point and one $\q$ point. Performing a nonorientable band move on $U$ with normal Euler number $\pm 1$ gives two nonorientable cobordisms $\mathcal{S}_{\pm}$ from $\mathcal{U}$ to $\mathcal{U}$. From the $\delta$-grading shift formula (see \cite[Proposition 7.2.5]{fan2019unoriented}), we see that the grading shifts of $F_{\mathcal{S}_+}$ and $F_{\mathcal{S}_-}$ are $0$ and $-1$, respectively. Now, from the fact that the cobordism maps $F_{\mathcal{S}_{\pm}}$ induced by $\mathcal{S}_{\pm}$ become quasi-isomorphisms after localizing by $U^{-1}$ (which follows directly from \cite[Proposition 5.4]{gong2021non}), we get
\[
F_{\mathcal{S}_+} = \mathbf{id} \text{ and } F_{\mathcal{S}_-} = U\cdot\mathbf{id}.
\]
\end{exmp}

In this section, we will use not only (disoriented) cobordisms in $S^3 \times I$, but also cobordisms in $B^4$. They can be seen as disoriented cobodisms from a nontrivial link to the empty link, or from the empty link to a nontrivial link. To define cobordism maps for such cobordisms, given a disoriented cobordism $(S,a)$ in $B^4$, we choose any $p$ in the interior of $a$ and its sufficiently small neighborhood $N(p)$ in $B^4$, so that $S\cap N(p)$ is a boundary-parallel disk in $B^4$ and $a\cap N(p)$ is a properly embedded arc on $S\cap N(p)$. We then consider $(S\backslash N(p),a\backslash N(p))$, which is now a disoriented cobordism in $S^3\times I$, from some disoriented link $L$ to the unknot. Then the map $F_{(S,a)}$ induced by $(S,a)$ can be naturally defined as the disoriented cobordism map $F_{(S\backslash N(p),a\backslash N(p))}$ induced by the punctured cobordism $(S\backslash N(p),a\backslash N(p))$ by \Cref{fansthm}. We can similarly define the opposite-direction cobordism map $F_{(-S,a)}$. Note that this definition is compatible with the birth/death maps of unknots in Zemke's oriented link Floer TQFT.

\begin{prop}
For any disoriented cobordism $(S,a)$ in $B^4$, the chain homotopy class of the map $F_{(S,a)}$ depends only the isotopy class of $(S,a)$. The same statement also holds for $(-S,a)$.
\end{prop}
\begin{proof}
Choose any two points $p,q$ in the interior of $a$ and choose sufficiently small neighborhoods $N(p)$ and $N(q)$ of $p$ and $q$, respectively, so that $N(p)\cap N(q)=\emptyset$ and $S\cap N(p)$. We claim that $F_{(S\backslash N(p),a\backslash N(p))}=F_{(S\backslash N(q),a\backslash N(q))}$; the desired result follows directly from this claim.
To prove the claim, we first observe that the case when $S$ is orientable is already covered by \Cref{oriented-equivalent}. In the general case when $S$ can be non-orientable, choose any smoothly embedded ball $B\subset B^4$ such that $N(p)\cup N(q)\subset B$, $S\cap B$ is orientable, and $\partial B$ intersects transversely with both $S$ and $a$. Since $(S\backslash B,a\backslash B)$ and $(S\cap (B\backslash N(p)),a\cap (B\backslash N(p)))$ are disoriented cobordisms in $S^3 \times I$, we have 
\[
F_{(S\backslash N(p),a\backslash N(p))} \simeq F_{(S\backslash B,a\backslash B)} \circ F_{(S\cap (B\backslash N(p)),a\cap (B\backslash N(p))} \simeq F_{(S\backslash B,a\backslash B)} \circ F_{(S\cap B,a\cap B)}.
\]
Similarly, we also have $F_{(S\backslash N(q),a\backslash N(q))} \simeq F_{(S\cap B,a\cap B)}$. The claim follows.
\end{proof}

We now introduce involutive knot Floer homology. Given an oriented knot $K$ in $S^3$, we can represent it as a doubly-pointed Heegaard diagram $H=\left(\Sigma,\mathbf{\alpha},\mathbf{\beta},z,w\right)$, and suppose that $K$ lies on the Heegaard surface $\Sigma$. We have a canonical chain skew-isomorphism
\[
\eta \colon CFK(H)\rightarrow CFK(\bar{H}),
\]
where $\bar{H}$ denotes the Heegaard diagram $(-\Sigma,\mathbf{\beta},\mathbf{\alpha},w,z)$.  Here, the prefix `skew' means that we are considering maps which intertwine the actions of $U$ and $V$ on the domain with the actions of $V$ and $U$ on the codomain.

We then consider the half-twist map
\[
\tau \colon CFK(\bar{H})\rightarrow CFK(\tau(\bar{H}))
\]
which corresponds to the self-diffeomorphism $(S^3,K,z,w)\rightarrow (S^3,K,w,z)$ induced by rotating a neighborhood of $K$ (in $\Sigma$) halfway around $K$ in the direction of the orientation induced by $K$. Here, $\tau(\bar{H})$ denotes the Heegaard diagram $(-\Sigma,\tau(\mathbf{\beta}),\tau(\mathbf{\alpha}),z,w)$, where $\tau(\mathbf{\beta})$ and $\tau(\mathbf{\alpha})$ are the images of $\mathbf{\beta}$ and $\mathbf{\alpha}$ under the given ``half-twist'' self-diffeomorphism. Now, the doubly-pointed Heegaard diagrams $\tau(\bar{H})$ and $H$ encode the same oriented knot, so the naturality of link Floer homology \cite{juhasz2018naturality} tells us that there is a chain homotopy equivalence
\[
n \colon CFK(\tau(\bar{H}))\rightarrow CFK(H),
\]
which is well-defined up to chain homotopy. Composing these three maps gives a homotopy self skew-equivalence
\[
\iota_K =n\circ \tau\circ \eta,
\]
whose chain homotopy class depends only on the oriented isotopy class of $K$. It is known \cite{zemke2019connected} that $\iota_K$ is not an involution in general, but rather an order $4$ map which satisfies
\[
\iota_K ^2 \sim 1+\Phi\Psi,
\]
where $\Phi$ and $\Psi$ are the basepoint actions, which can be easily computed as the formal partial derivatives of the chain differential with respect to the formal variables $U$ and $V$, respectively. Note that, since the truncation from $CFK$ to $\uCFK$ identifies $U$ and $V$ into a single variable $U$, we see that $\iota_K$ induces a $\delta$-grading-preserving $\mathbb{F}_2[U]$-equivariant chain homotopy self-equivalence
\[
\iota_K \colon \uCFK(S^3,K) \rightarrow \uCFK(S^3,K),
\]
whose chain homotopy class depends on the oriented isotopy class of $K$. It is clear that, if we choose the opposite orientation of $K$, the map $\iota_K$ is replaced by its homotopy inverse $\iota_K ^{-1}$.

Since $\iota_K$ is now a $U$-equivariant chain homotopy equivalence of $\mathbb{F}[U]$-modules rather than a skew-equivariant chain homotopy equivalence of $\mathbb{F}[U,V]$-modules (which is the case in the oriented setting), we can define new knot invariants as follows.

\begin{defn}
Given an unoriented knot $K$, choose either orientation on $K$ and define the \emph{involutive unoriented knot Floer chain complex} $\uCFKI(K)$ as the mapping cone of $1+\iota_K \colon \uCFK(S^3,K)\rightarrow \uCFK(S^3,K)$, which is a chain complex of $\mathbb{F}_{2}[U,Q]/(Q^2)$-modules. This complex is endowed with a grading which is uniquely determined from the $\delta$-grading on $\uCFK(S^3,K)$ by the rule that the grading of $Q$ is $-1$. The homology of $\uCFKI(K)$, called the \emph{involutive unoriented knot Floer homology}, is denoted $\uHFKI(K)$.
\end{defn}

We note our grading conventions for the mapping cone differ from the conventions used in Hendricks-Manolescu's definition of involutive Heegaard Floer homology~\cite{hendricks2017involutive}.

\begin{rem}
If we denote the differential on $CFK(S^3,K)$ as $\partial_K$, then the differential $\partial$ on $\uCFKI(K)$ is given by $\partial(a+Qb) = \partial_K(a)+Q(a+\iota_K(a)+\partial_K(b))$.
\end{rem}

It is easy to see that the above definition is independent of the choice of orientation on $K$.
\begin{lem}
\label{mappingcone}
Let $C,D$ be chain complexes and $g \colon C\rightarrow D$ be a chain homotopy equivalence. Then for any chain map $f \colon C\rightarrow C$, the mapping cones of $f$ and $g\circ f$ are chain homotopy equivalent.
\end{lem}
\begin{proof}
The differentials on the mapping cones of $f$ and $g\circ f$ are given by $\partial_{M(f)}(a,b)=(\partial a,f(a)+\partial b)$ and $\partial_{M(g\circ f)}(a,b)=(\partial a,g(f(a))+\partial b)$, respectively. Consider the maps $F:M(f)\rightarrow M(g\circ f)$ and $G:M(g\circ f)\rightarrow M(f)$, defined as $F(a,b)=(a,g(b))$ and $G(a,b)=(a,h(b))$, where $h$ is a homotopy inverse of $g$. Then we have 
\[
(\partial_{M(g\circ f)}F+F\partial_{M(f)})(a,b)=(\partial a,(g\circ f)(a)+\partial g(b))+(\partial a,g(f(a)+\partial b))=0,
\]
so $F$ is a chain map, and similarly $G$ is also a chain map. Now, if we denote a null-homotopy of $h\circ g - \mathbf{id}$ as $H$, then we have 
\[
G\circ F-\mathbf{id}=(0,h\circ g-\mathbf{id})=\partial_{M(f)}\circ (0,H) - (0,H)\circ \partial_{M(f)},
\]
so $G\circ F \sim \mathbf{id}$ and similarly $F\circ G \sim \mathbf{id}$. Therefore $F$ is a chain homotopy equivalence.
\end{proof}

\begin{thm}
\label{welldef}
Given a knot $K$, the chain homotopy equivalence class of $\uCFKI(K)$ as a chain complex of $\mathbb{F}_2[U,Q]/(Q^2)$-modules depends only on the isotopy class of $K$.
\end{thm}
\begin{proof}
Choose an orientation on $K$, which then defines $\iota_K$ on $\uCFKI(K)$. We have to prove that the mapping cone of $1+\iota_K$ is chain homotopy equivalent to the mapping cone of $1+\iota_K^{-1}$. But since
\[
1+\iota_K^{-1}=\iota_K^{-1}\circ (1+\iota_K)
\]
and $\iota_K$ is a chain homotopy self-equivalence of $\uCFKI(K)$, \Cref{mappingcone} implies that the two mapping cones are indeed chain homotopy equivalent.
\end{proof}

The following lemma tells us about the structure of $\uHFKI$. For simplicity, we denote the localization by introducing a formal inverse of $U$ as $U^{-1}$.
\begin{lem}
\label{structurelem}
We have $U^{-1}\uHFKI(K)\simeq \mathbb{F}_2[U^{\pm 1},Q]/(Q^2)$.
\end{lem}
\begin{proof}
Since $U^{-1}\uCFK(S^3,K)$ is the truncation of $(U,V)^{-1}CFK(S^3,K)$ by the identification $U=V$ and $(U,V)^{-1}CFK(S^3,K)\simeq (U,V)^{-1}\mathbb{F}_{2}[U,V]$, we know that $U^{-1}\uCFK(S^3,K)\simeq U^{-1}\mathbb{F}_{2}[U]\simeq \mathbb{F}_{2}[U^{\pm 1}]$. Thus $U^{-1}\iota_K$, the localization of $\iota_K$ at $U$, is chain homotopic to the identity map. Now the lemma follows from the fact that a localization of a mapping cone is chain homotopy equivalent to the mapping cone of a localized map.
\end{proof}

The above lemma, combined with \Cref{welldef}, allows us to make the following definition.

\begin{defn}
Given a knot $K$, we define two knot invariants $\bar{\upsilon}(K)$ and $\ubar{\upsilon}(K)$ as follows:
\[
\begin{split}
    \ubar{\upsilon}(K) &= \max \left\{ r \,\vert\, \exists x\in \uHFKI(K),\,\mathbf{gr}_{\delta}(x)=r,\,\forall n,\,U^n x\notin \text{Im}(Q) \right\}, \\
    \bar{\upsilon}(K) &= \max \left\{ r \,\vert\, \exists x \in \uHFKI(K),\,\mathbf{gr}_{\delta}(x)=r,\,\forall n,\,U^n x\ne 0;\exists m\ge 0 \text{ s.t. } U^m x\in \text{Im}(Q)\right\} +1,
\end{split}
\]
where $\mathbf{gr}_{\delta}(x)$ denotes the (normalized) $\delta$-grading of a homogeneous element $x$.
\end{defn}

We note that the grading shifts in the above definition differ from the corresponding grading shifts in the definitions of $\underline{d}$ and $\overline{d}$ given by Hendricks-Manolescu~\cite[Section 5]{hendricks2017involutive} on account of the different grading convention for mapping cones we used in our definition of involutive unoriented knot Floer homology.

As mentioned in the introduction, the invariants $\ubar{\upsilon}$ and $\bar{\upsilon}$ coincide with the values at $t=1$ of the involutive Upsilon invariants $\ubar{\Upsilon}_t$ and $\bar{\Upsilon}_t$ defined by Hogancamp and Livingston. It thus follows from \cite[Proposition 11 and Theorem 12]{hogancamp2017involutive} that $\ubar{\upsilon}$ and $\bar{\upsilon}$ are (smooth) concordance invariants and $\bar{\upsilon}(K) \ge \upsilon(K) \ge \ubar{\upsilon}(K)$.

\begin{exmp}\label{ex:figeight}
Consider the figure-eight knot $K=4_1$. Its unoriented knot Floer complex has generators $a,b,c,d,x$, whose gradings are all $0$, and where the differential is given by 
\[
    \partial a = U(b+c), \,
    \partial b = \partial c = Ud, \,
    \partial x = \partial d = 0.
\]
The involution $\iota_K$, which was computed in \cite[Section 8.2]{hendricks2017involutive}, is given by 
\[
    \iota_K(a) = a+x, \,
    \iota_K(b) = c, \,
    \iota_K(c) = b, \,
    \iota_K(d) = d, \,
    \iota_K(x) = x+d.
\]
We will calculate $\bar{\upsilon}(K)$ and $\ubar{\upsilon}(K)$ by explicitly computing the cycles of $\uCFKI(K)$. Recall that the differential of $\uCFKI(K)$ is given by $\partial (s+Qt)=\partial s+Q(s+\iota_K(s)+\partial t)$. Hence an element $a+Qb\in \uCFKI(K)$ is a cycle if and only if $a$ is a cycle in $\uCFK(S^3,K)$ and $(1+\iota_K)(a)=\partial b$. 

The space of cycles in $\uCFK(S^3,K)$ is generated over $\mathbb{F}_{2}[U]$ by $x$, $b+c$, and $d$, whose images under the action of $1+\iota_K$ are given by $d$, $0$, and $0$, respectively. Hence the space of cycles in $\uCFKI(K)$ is generated over $\mathbb{F}_{2}[U,Q]/(Q^2)$ by the following cycles:
\[
Ux+Qb,b+c,\,d,\, Qx.
\]
Let $\partial'$ denote the differential on $\uCFK(K)$. Since $ \partial'(b)=Ud+Q(b+c)$, $Ud$ is in the image of $Q$ in $\uHFKI(K)$. Likewise $\partial'(a+Qx)=U(b+c)+Q(x)$, so that $U(b+c)$ is in the image of $Q$ in $\uHFKI(K)$.

On the other hand, if $Ux+Qb$ contained in the image of $Q$  in $\uHFKI(K)$ for some $n\ge 0$, then we should have 
\[
U^{n}(Ux+Qb) = Qw+\partial_1(y+Qz)
\]
for elements $w,y,z\in \uCFK(S^3,K)$. It follows that $\partial y=U^{n+1}(x)$, a contradiction. Hence the homology class of $U^n(Ux+Qb)$ is not contained in the image of $Q$ for any $n\ge 0$. Finally observe that if $U^n(b+c)=0$ in homology then there exist elements $z,y\in\uCFK(K)$, $n\geq 0$ such that $\partial z=U^n(b+c)$, $(1+\iota_K)(x)+\partial y=0$. It follows that $z=U^{n-1}a$, whence $U^{n-1}a+\partial y=0$, a contradiction.

Now, using the definition of $\bar{\upsilon}$ and $\ubar{\upsilon}$, we have
\[
\begin{split}
    \bar{\upsilon}(K) &= \mathbf{gr}_{\delta}(b+c)+1 = 1 \qquad \text{and} \\
    \ubar{\upsilon}(K) &= \mathbf{gr}_{\delta}(Ux+Qb) = -1. 
\end{split}
\]
Together with \Cref{mobiusthm}, which will be proven in this section, we recover the fact that the figure-eight knot does not bound a smooth Möbius band in $B^4$, which was proven in the topological category in \cite{viro}. Note that, since the figure-eight knot is rationally slice, we have $\upsilon(K)=\sigma(K)=0$.

In general, via a straightforward computation using model complexes as given in \cite[Section 8.3]{hendricks2017involutive}, it is easy to see that the involutive upsilon invariants of a thin knot $K$ (which includes all quasialternating knots) are given by 
\begin{align*}
    \bar{\upsilon}(K)=\upsilon(K)=\ubar{\upsilon}(K)=0 &\text{ if } \sigma(K)+4\arf(K) \equiv 0 \mod 8, \\
    \bar{\upsilon}(K)=1,\,\upsilon(K)=0,\,\ubar{\upsilon}(K)=-1 &\text{ if } \sigma(K)+4\arf(K) \equiv 4 \mod 8.
\end{align*}
We thus see that a thin knot $K$ does not bound a smooth M\"{o}bius band in $B^4$ if $\sigma(K)+4\arf(K) \equiv 4 \mod 8$, which is the same as Yasuhara's obstruction from \Cref{lem:yasuhara}.
\end{exmp}

Now we move on to developing involutive obstructions and bounds on the first Betti number for nonorientable slice-surfaces bounding a given knot. Recall that, given a knot $K$ with two prescribed points $p,q\in K$, the map $\iota_K$ on the unoriented knot Floer complex is defined by composing the canonical conjugation map
\[
\eta_K  \colon \uCFK(K,p,q)\rightarrow \uCFK(K,q,p)
\]
together with the ``half twist'' map (which depends on an orientation of $K$)
\[
\tau_K \colon \uCFK(K,q,p)\rightarrow \uCFK(K,p,q).
\]
Note that the half twist map can be seen as a cobordism map induced by the ``half twist cobordism'', as shown in Figure \ref{fig:halftwistcobordism}. In particular, it is not a skew-isomorphism anymore; it is a genuine isomorphism, as we are now working with $\uCFK$, where the actions of $U$ and $V$ coincide.

\begin{figure}[htbp]
    \centering
    \includegraphics[width=0.3\textwidth]{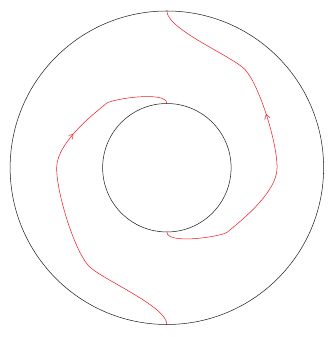}
    \caption{The cobordism which induces the half twist map $\tau_K$. Note that the underlying surface is $K\times I$.}
\label{fig:halftwistcobordism}
\end{figure}

For more generality, we extend the definition of canonical conjugation map to disoriented links. Given any disoriented link $\mathcal{L}=(L,\p,\q,\bfl)$, we also have a canonically defined conjugation map
\[
\eta_{\mathcal{L}} \colon \uCFL(\mathcal{L})\rightarrow \uCFL(\bar{\mathcal{L}}),
\]
where $\bar{\mathcal{L}}=(L,\q,\p,\bfl)$. 

\begin{lem}
\label{conjugationcomm}
Let $\mathcal{S}=(S,a)$ be a disoriented cobordism between disoriented links $\mathcal{L}_1$ and $\mathcal{L}_2$. Then we have
\[
F_{(S,-a)} \circ \eta_{\mathcal{L}_1} \sim \eta_{\mathcal{L}_2} \circ F_{(S,a)},
\]
where $-a$ denotes $a$ endowed with the reverse orientation.
\end{lem}
\begin{proof}
Recall that the cobordism maps in Fan's unoriented link Floer TQFT are computed by decomposing a given cobordism into elementary disoriented link cobordisms and composing the maps corresponding to each of these pieces. Among all possible pieces, we only have to consider nonorientable saddle moves, since this Lemma is verified in all other cases in the proof of \cite[Theorem 1.3]{zemke2019connected} (for compatible decorated link cobordisms). Since the cobordism map for a saddle move is defined via a count of pseudoholomorphic triangles on certain simple Heegaard triple-diagrams -- see the proof of \cite[Theorem 5.1.4]{fan2019unoriented} -- they homotopy-commute with conjugation maps.
\end{proof}

We now have to check that the half twist map $\tau_K$ also homotopy-commutes with (disoriented) cobordism maps. When the given cobordism is a M{\" o}bius band, then the commutativity is very easy to observe.
\begin{lem}
\label{Mobiuscomm}
Let $\mathcal{S}=(S,a)$ be a disoriented cobordism from a disoriented knot $K=(K,\p,\q,\bfl)$ to the empty link, where $\p$ and $\q$ both consist of a single point. Suppose that $S$ is a smoothly embedded M{\" o}bius band and $a$ is a simple arc. Then, for any choice of orientation on $K$, we have $F_{(S,-a)} \circ \tau_K \sim F_{(S,a)}$, where $-a$ denotes $a$ endowed with the reverse orientation. Moreover, if $\bar{S}$ denotes the cobordism from the empty link to $K$ given by flipping $S$ upside-down, then we have $\tau_K \circ F_{(\bar{S},a)} \sim F_{(\bar{S},-a)}$.
\end{lem}
\begin{proof}
This follows directly from Figure \ref{fig:mobiuscomm}.
\end{proof}

\begin{figure}[htbp]
    \centering
    \includegraphics[width=0.7\textwidth]{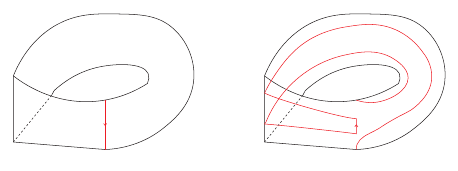}
    \caption{Left, the cobordism $\mathcal{S}$. Right, the cobordism representing $\tau_K \circ (\bar{S},-a)$. Note that the cobordism on the right is isotopic to the one on the left.}
\label{fig:mobiuscomm}
\end{figure}

With \Cref{Mobiuscomm} at hand, we can now prove \Cref{mobiusthm}.
\begin{proof}[Proof of \Cref{mobiusthm}]
Let $M$ be a smooth M{\" o}bius band in $B^4$ bounding $K$. Choose a properly embedded simple arc $a\subset M$ such that $M\backslash a$ is a rectangle. Then $(M,a)$ is a disoriented cobordism from $K$ to the empty link and $(-M,a)$ is a disoriented cobordism from the empty link to $K$. By choosing a interior point $p\in a\cap \text{int}(M)$ and its small neighborhood $N(p)\subset B^4$, we can consider two disoriented cobordisms $(M\backslash N(p),a\backslash N(p))$ and $(-M\backslash N(p),a\backslash N(p))$, which are cobordism from $K$ to unknot and unknot to $K$, respectively; note that $F_{(M\backslash N(p),a\backslash N(p))}=F_{(M,a)}$ and $F_{(-M\backslash N(p),a\backslash N(p))}=F_{(-M,a)}$. Then, using \cite[Proposition 5.4]{gong2021non} on the punctured cobordism $(M\backslash N(p),a\backslash N(p))$, we see that there exists an orientable disoriented cobordism $T$ with first Betti number $1$ from $K$ to $K$, shown in Figure \ref{fig:cobordismT}, such that 
\[
U^c \cdot F_T \circ F_{(-M,a)} \circ F_{(M,a)} \sim U^{b-m+1} \cdot \mathbf{id},
\]
 where $m,b,c$ are the numbers of local minima, saddles, and local maxima of $M$; note that $b-m-c=1$ since $M$ is a Mobius band. Since the underlying space of $T$ is an orientable cobordism, $F_T$ is homotopic to the $U=V$ truncation of the cobordism map induced by $T$ (viewed as a compatible decorated link cobordism) with some orientation in Zemke's oriented link Floer TQFT. Theorem C of \cite{zemke2019link} implies that $U^{-1}F_T \sim U\cdot \mathbf{id}$ -- where here $U^{-1}F_T$ is the localization of $F_T$ at $U$ -- so we see that the localized map $U^{-1}(F_{(-M,a)} \circ F_{(M,a)})$ is a quasi-isomorphism. Since $U^{-1}\uHFK(S^3,K)\simeq \mathbb{F}_2[U^{\pm 1}]$, we see that the localized maps $U^{-1}F_{(-M,a)}$ and $U^{-1}F_{(M,a)}$ are both chain homotopic to the identity. Also, since $U^{-1}F_T \sim U\cdot \mathbf{id}$, we know that $F_T$ shifts the $\delta$-grading by $-1$. Hence, if we denote the $\delta$-grading shifts of $F_{(-M,a)}$ and $F_{(M,a)}$ by $\mathbf{gr}_\delta (F_{(-M,a)})$ and $\mathbf{gr}_\delta (F_{(M,a)})$, then we have 
\[
\mathbf{gr}_\delta (F_{(-M,a)})+\mathbf{gr}_\delta (F_{(M,a)})=-1.
\]

\begin{figure}[htbp]
    \centering
    \includegraphics[width=0.3\textwidth]{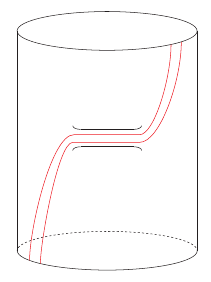}
    \caption{The orientable cobordism $T$ of first Betti number $1$.}
\label{fig:cobordismT}
\end{figure}

Choose either orientation on $K$. \Cref{conjugationcomm} and \Cref{Mobiuscomm} imply that  $F_{(M,a)}$ and $F_{(-M,a)}$ satisfy the following conditions:
\[
F_{(M,a)} \circ \iota_K  \sim F_{(M,a)}, \,\iota_K\circ F_{(-M,a)}\sim F_{(-M,a)}.
\]
Hence, by choosing relevant chain homotopies, we get maps between involutive homology:
\[
\begin{split}
F \colon & \uHFKI(K)\rightarrow \mathbb{F}_2[U,Q]/(Q^2), \\
G \colon & \mathbb{F}_2[U,Q]/(Q^2) \rightarrow \uHFKI(K),
\end{split}
\]
where the grading shifts of $F$ and $G$ are given by $\mathbf{gr}_\delta (F_{(M,a)})$ and $\mathbf{gr}_\delta (F_{(-M,a)})$, respectively. Note that $F$ and $G$ become quasi-isomorphisms when localized by $U^{-1}$.

Now, from the definition of $\bar{\upsilon}$, we know that there exists an element $x\in \uHFKI(K)$ such that the following conditions are satisfied.
\begin{itemize}
    \item The $\delta$-grading of $x$ is $\bar{\upsilon}(K)-1$.
    \item For any $n\ge 0$, we have $U^n x\ne 0$.
    \item $U^m x = Qy$ for some $m\ge 0$ and $y\in \uHFKI(K)$.
\end{itemize}
Consider the element $F(x)\in \mathbb{F}_2[U,Q]/(Q^2)$. Since the localized map $U^{-1}F$ is chain homotopic to $U$ times the identity map and the homology class of $x$ generates the tower given by the image of $Q$ inside the ($U$-)localized homology, we know that $F(x)\ne 0$. Also, we have $U^i F(x)=QF(y)$, so $F(x)=QU^i$ for some $i\ge 0$. Hence the $\delta$-grading of $F(x)$ is at most $-1$, so we get 
\[
\bar{\upsilon}(K)\le -\mathbf{gr}_{\delta}(F_{(M,a)}).
\]
We now consider the element $1\in \uCFKI(\text{unknot})=\mathbb{F}_2[U,Q]/(Q^2)$. Then $G(1)$ is a cycle in $\uCFKI(K)$. Suppose that there exists some $m\ge 0$ such that the homology class of $U^m G(1)$ is a multiple of $Q$. Then there exist a cycle $d\in \uCFK(S^3,K)$ and chains $e,f\in \uCFK(S^3,K)$ such that $$G(U^m )=U^m G(1)=Qd+\partial (e+Qf)=\partial e+Q(d+e+\iota_K(e)+\partial f).$$ Since $G(U^m )$ is of the form $F_{(-M,a)}(U^m )+Qc$ for some $c\in \uCFK(S^3,K)$, we must have $F_{(-M,a)}(U^m )=\partial e$, \ie the homology class $[F_{(-M,a)}(U^m )]\in \uHFK(K)$ vanishes. But this contradicts the fact that the localized map $U^{-1}F_{(-M,a)}$ is chain homotopic to identity. Hence we deduce that $U^m G(1)\notin \text{Im}(Q)$ for any $m\ge 0$, which implies that 
\[
\ubar{\upsilon}(K)\ge \mathbf{gr}_{\delta}(F_{(-M,a)}).
\]
Therefore we deduce that
\[
\bar{\upsilon}(K)-\ubar{\upsilon}(K)\le -(\mathbf{gr}_\delta (F_{(-M,a)})+\mathbf{gr}_\delta (F_{(M,a)}))=1.
\]
\end{proof}

We move on to prove an involutive bound on the nonorientable genus. For this we require the following lemmas.
\begin{lem}[{\cite[Proposition 6.6.2]{fan2019unoriented}}]
Let $S_i =(\Sigma,a_i)$, $i=1,2,3$, be disoriented link cobordisms on the same underlying unoriented surface $\Sigma$. Suppose that $S_1,S_2,S_3$ are related by a \emph{bypass move}, which is defined as in Figure \ref{fig:bypassrelation}. Then we have $F_{S_1}+ F_{S_2}+F_{S_3} \sim 0$.
\end{lem}

\begin{figure}[htbp]
    \centering
    \includegraphics[width=0.7\textwidth]{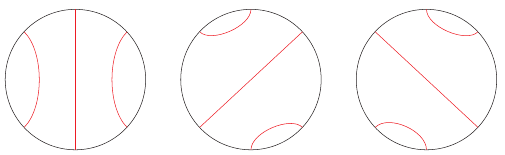}
    \caption{We say that three disoriented cobordisms are related by a \emph{bypass move} if they differ only inside a disk as drawn. See \cite[Definition 7.1.5]{fan2019unoriented} for more details.}
\label{fig:bypassrelation}
\end{figure}

\begin{lem}
\label{closedcomp}
Let $S=(S,a)$ be a disoriented link cobordism from a disoriented link $L$ to the empty link. Suppose that the one manifold $a$ contains a closed component $c$ which is \emph{two-sided}, \ie admits an orientable neighborhood in $S$. Then $U\cdot F_S \sim 0$ and $U\cdot F_{\bar{S}}\sim 0$, where $\bar{S}$ is given by flipping $S$ upside-down.
\end{lem}
\begin{proof}
Choose a disk $D$, smoothly embedded in the interior of $B^4$, which intersects $S$ transversely in its interior and satisfies $\partial D=c$. Choose a small neighborhood $N_D$ of $D$. By perturbing $a$, we may assume that $S_1 = (S\backslash N_D,a\backslash N_D)$ and $S_2 = (S\cap N_D,a\cap N_D)$ are disoriented cobordisms and $c\subset N_D$. Note that, since $S=S_2 \circ S_1$, we have $F_S = F_{S_2} \circ F_{S_1}$.

Since $c$ is two-sided and the interior of $D$ intersects $S$ in finitely many points, $S\cap N_D$ is orientable. Thus, once again replacing the disoriented cobordism structure on $S_1$ and $S_2$ with the structure of compatible decorated link cobordisms, by Lemma 5.2 of \cite{juhasz2020concordance} and Remark \ref{oriented-equivalent}, we deduce that $U\cdot F_{S_2}\sim 0$. Therefore $U\cdot F_S = (U\cdot F_{S_2})\circ F_{S_1} \sim 0$. The proof for $\bar{S}$ is the same.
\end{proof}

Now, we can prove that the disoriented cobordism maps induced by surfaces of higher first Betti number homotopy-commute with $\tau_K$ up to multiplication by $U$.
\begin{lem}
\label{highergenuscomm}
Let $\mathcal{S}=(S,a)$ be a disoriented cobordism from a disoriented knot $\mathcal{K}=(K,\p,\q,\bfl)$ to the empty link, where $\p$ and $\q$ consist of a single point. Suppose that $S$ is a smoothly embedded nonorientable surface and $a$ is a simple arc. Then, for any choice of orientation on $K$, we have $U\cdot (F_{(S,a)} \circ \tau_K) \sim U \cdot F_{(S,-a)}$, where $-a$ denotes $a$ endowed with the reverse orientation. Also, if $\bar{S}$ denotes the cobordism from the empty link to $K$ given by flipping $S$ upside-down, then we have $U\cdot (\tau_K \circ F_{(\bar{S},a)}) \sim U\cdot F_{(\bar{S},-a)}$.
\end{lem}
\begin{proof}
We will only provide a proof for $(S,a)$; the proof for $(\bar{S},a)$ is the same. 

We first assume that the first Betti number of $S$ is odd. Then we can consider $S$, as an abstract surface with boundary, as a connected sum of a M\"{o}bius band and a closed orientable surface, and $a$ as an arc on $S$ which is shown on the left in \Cref{fig:highergenus}. Choose an orientation on $K$ and denote by $C$ the disoriented cobordism from $\mathcal{K}$ to $(K,\q,\p,\bfl)$ given by a half-twist along the given orientation on $K$; note that the cobordism map induced by $C$ is $\tau_K$. Then $\mathcal{S} \circ C$ is isotopic to the disoriented cobordism $\mathcal{S}^{\prime}=(S,b)$, where $b$ is the arc shown on the right in \Cref{fig:highergenus}. 

\begin{figure}[htbp]
    \centering
    \includegraphics[width=0.9\textwidth]{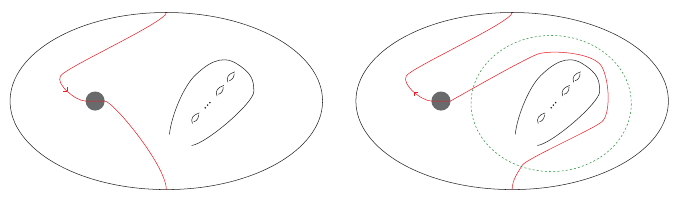}
    \caption{Left, the cobordism $\mathcal{S}$. Right, the cobordism $\mathcal{S}^{\prime}$. The region enclosed by the green dashed curve is $S_o$. The grey circles denote crosscaps, \ie M\"{o}bius band attached along its boundary.}
\label{fig:highergenus}
\end{figure}

The oriented arcs $-a$ and $b$ differ in an oriented subsurface $S_o$ of $S$. Denote the first Betti number of $S_o$ by $n$. Then one can consider a sequence of properly embedded (oriented) arcs $-a=a_0,a_1,\cdots,a_n=b$, as drawn in \Cref{fig:arcsinS0}, such that the following conditions are satisfied:
\begin{itemize}
    \item $a_0,\cdots,a_n$ differ only in $S_o$,
    \item For each $i$, $S_o \backslash a_i$ consists of a surface $A_i$ of genus $i$ and a surface $B_i$ of genus $n-i$, where $A_i$ lies on the right of $a_i$,
    \item $A_0 \subset A_1 \subset \cdots \subset A_n$.
\end{itemize}
Note that, since $a_i$ is oriented and divides $S_o$ into two components, we have a well-defined notion of a component lying on the left side and the right side of $a_i$. Since we have \Cref{closedcomp} and a bypass relation for Fan's TQFT, we can directly apply the arguments used in the proof of \cite[Lemma 5.3 and Proposition 5.5]{juhasz2020concordance} to show that $U \cdot F_{(S,a_{i-1})} \sim U \cdot F_{(S,a_{i})}$ for all $i$. Hence we get $U\cdot (F_{\mathcal{S}} \circ \tau_K) \sim U \cdot F_{(S,-a)}$.

We now consider the remaining case when the first Betti number of $S$ is even. Then, by \Cref{unknottedmobius}, we see that taking a connected sum of $\mathcal{S}$ with a unknotted $\mathbb{RP}^2$ with normal Euler number 2 gives a disoriented cobordism of odd first Betti number which has the same induced cobordism map as $\mathcal{S}$. This reduces the problem to the case of odd first Betti number, so we get $U\cdot (F_{\mathcal{S}} \circ \tau_K) \sim U \cdot F_{(S,-a)}$, as desired.
\end{proof}

\begin{figure}[htbp]
    \centering
    \includegraphics[width=0.5\textwidth]{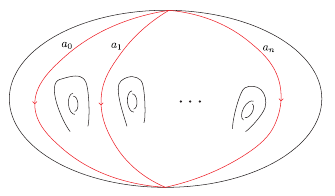}
    \caption{The arcs $a_0,a_1,\cdots,a_n$ on the orientable surface $S_o$.}
\label{fig:arcsinS0}
\end{figure}

We are now able to prove \Cref{highergenusthm}. Our proof is similar to that of Equation 18 in~\cite[Theorem 1.7]{juhasz2020concordance}.
\begin{proof}[Proof of \Cref{highergenusthm}]
Let $S$ be a smoothly embedded nonorientable surface of first Betti number $g$ in $B^4$ which bounds $K$. Suppose that $S$ consists of $M$ local maxima, $b$ (possibly nonorientable) saddles, and $m+1$ local minima. Then, by applying \cite[Proposition 5.5]{gong2021non} as in the proof of \Cref{mobiusthm}, there exists an orientable disoriented cobordism $T$ of first Betti number $3$, given as shown in Figure \ref{fig:cobordismT}, and a properly embedded oriented arc $a\subset S$, such that $S\backslash a$ is orientable and 
\[
U^M \cdot (F_T \circ F_{(-S,a)} \circ F_{(S,a)}) \sim U^{b-m+1} \cdot \mathbf{id}.
\]
Following the proof of \Cref{mobiusthm} tells us that the localized maps $U^{-1}F_{(-S,a)}$ and $U^{-1}F_{(S,a)}$ are chain homotopic to the identity map, and if we denote the $\delta$-grading shifts of $F_{(-S,a)}$ and $F_{(S,a)}$ by $\mathbf{gr}_\delta (F_{(-S,a)})$ and $\mathbf{gr}_\delta (F_{(S,a)})$, then we have 
\[
\mathbf{gr}_\delta (F_{(-S,a)})+\mathbf{gr}_\delta (F_{(S,a)})=M+m-b=g.
\]
However, since $F_{(S,a)}$ and $F_{(-S,a)}$ may not homotopy-commute with $\iota_K$, they do not induce maps between $\uHFKI(K)$ and $\uHFKI(\text{unknot})=\mathbb{F}_2[U,Q]/(Q^2)$. Fortunately, from \Cref{conjugationcomm} and \Cref{highergenuscomm}, we see that 
\[
\begin{split}
    U\cdot (\iota_K \circ F_{(S,a)}) &\sim U\cdot F_{(S,a)}, \\
    U\cdot (\iota_K \circ F_{(-S,a)}) &\sim U\cdot F_{(-S,a)},
\end{split}
\]
so $U\cdot F_{(S,a)}$ and $U\cdot F_{(-S,a)}$ induce maps between involutive homology
\[
\begin{split}
F\colon & \uHFKI(K)\rightarrow \mathbb{F}_2[U,Q]/(Q^2), \\
G\colon & \mathbb{F}_2[U,Q]/(Q^2) \rightarrow \uHFKI(K),
\end{split}
\]
where the grading shifts of $F$ and $G$ are denoted by $\mathbf{gr}_\delta (F_{(S,a)})-1$ and $\mathbf{gr}_\delta (F_{(-S,a)})-1$, respectively.

Now, following the proof of \Cref{mobiusthm} gives
\[
\begin{split}
    \bar{\upsilon}(K) &\le 1-\mathbf{gr}_{\delta}(F_{(S,a)}), \\
    \ubar{\upsilon}(K) &\ge \mathbf{gr}_{\delta}(F_{(-S,a)}) -1.
\end{split}
\]
Furthermore, using the same argument for $\upsilon(K)$ implies
\[
\mathbf{gr}_{\delta}(F_{(-S,a)}) \le \upsilon(K) \le -\mathbf{gr}_{\delta}(F_{(S,a)}).
\]
The theorem then follows directly from the above inequalities.
\end{proof}

We now focus on the case of $L$-space knots. The knot Floer homology of $L$-space knots is very simple: if $K$ is an $L$-space knot, then $CFK(S^3,K)$ is a staircase complex which is completely determined by the Alexander polynomial of $K$, as shown in \cite[Theorem 1.2]{ozsvath2005knot}. Recall that the Alexander polynomial of $K$ always takes the form 
\[
\Delta_K(t)=(-1)^m + \sum_{i=1}^{m} (-1)^{m-i}\left(t^{n_i}+t^{-n_i}\right)
\]
for some sequence of positive integers $0<n_1<\cdots <n_m =g(K)$. The chain complex $CFK(S^3,K)$ is chain homotopy equivalent to a free chain complex generated by elements $x_0,x^{1}_1,x^{2}_1,\cdots,x^{1}_m,x^{2}_m$. Here, $x_0$ lies on the bigrading $(0,0)$ and either:

\begin{enumerate}

\item the bigradings of $x^{1}_{m-i}$ and $x^{1}_{m-i-1}$ differ by $(n_{m-i}-n_{m-i-1},0)$ if $i$ is even and $(0,n_{m-i}-n_{m-i-1})$ if $i$ is odd or;
\item the bigradings of $x^{1}_{m-i}$ and $x^{1}_{m-i-1}$ differ by $(n_{m-i}-n_{m-i-1},0)$ if $i$ is odd and $(0,n_{m-i}-n_{m-i-1})$ if $i$ is even.

\end{enumerate} Moreover, if $x^{1}_s$ lies on the bigrading $(i,j)$, then $x^{2}_s$ lies on the bigrading $(j,i)$. The differential on the staircase complex is given as in \cite[Figure 12]{hendricks2017involutive}.

The action of $\iota_K$ on such a staircase complex is very easy to describe: we have $\iota_K(x_0)=x_0$, $\iota_K(x^{1}_s)=x^{2}_s$, and $\iota_K(x^{2}_s)=x^{1}_s$. In particular, we have $\iota_K^2 \sim \mathbf{id}$. See \cite[Section 7]{hendricks2017involutive} for more details.

\begin{lem}
\label{Lspaceknots}
Let $K$ be an $L$-space knot. If the Alexander polynomial of $K$ is of the form $\Delta_K(t)= -1+\sum_{i=1}^{m} (-1)^{m-i}(t^{n_i}+t^{-n_i})$ for some sequence of positive integers $0<n_1<\cdots <n_m$, then 
\begin{align*}
    \bar{\upsilon}(K)=\upsilon(K) \ge \ubar{\upsilon}(K)+n_1.
\end{align*}
\end{lem}
\begin{proof}
The unoriented knot Floer chain complex of $L$-space knots can be written as a graded root, as described in \cite[Section 7]{gong2021non}. The involution $\iota_K$ then acts on the graded root as a reflection along the central axis, so we can consider $(\uCFK(S^3,K),\iota_K)$ as a symmetric graded root, which is defined in \cite[Definition 2.11]{dai2019involutive}. See Figure \ref{fig:symmetricgradedroot} for a pictorial description of a symmetric graded root induced by the action of $\iota_K$ on $\uCFK(S^3,K)$.

\begin{figure}[htbp]
    \centering
    \includegraphics[width=0.5\textwidth]{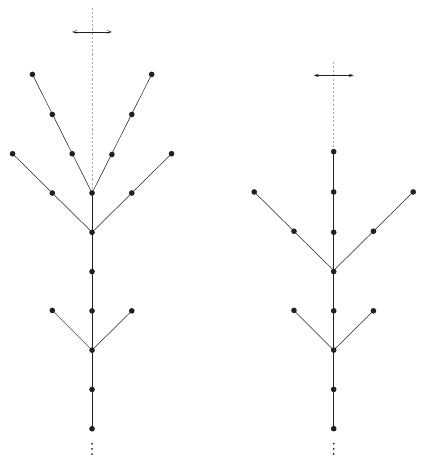}
    \caption{Left, the symmetric graded root representing $\uCFK(S^3 ,T_{6,7})$. Right, the symmetric graded root representing $\uCFK(S^3,T_{5,6})$. Note that the action of $\iota_K$ is given by reflection along the central axis, given by the gray dashed line. Note moreover that $\Delta_K(0)\neq -1$ for these two knots, so Lemma~\ref{Lspaceknots} does not apply.}
\label{fig:symmetricgradedroot}
\end{figure}

It is shown in \cite[Section 6]{dai2019involutive} that the invariants $d$ and $\overline{d}$ can be read off from a symmetric graded root. As $\upsilon$ and $\overline{\upsilon}$ are defined analogously to $f$ and $\overline{d}$, Dai and Manolescu's arguments likewise show that $\upsilon(K)$ can be read off from the given symmetric graded root as the grading of the highest leaf, and $\bar{\upsilon}(K)=\upsilon(K)$. Furthermore, $\ubar{\upsilon}(K)$ is the grading of the highest node which is left invariant by the action of $\iota_K$. Since the constant term of the Alexander polynomial of $K$ is $-1$, we have $\partial x_0 = U^{n_1}(x^1_1 + x^2_1)$ in $\uCFK(K)$, so $\ubar{\upsilon}(K)$ is represented by the cycle $U^{n_1} x^1_1 + Qx_0$ in $\uCFKI(K)$. Also, since $Qx_1$ is a cycle which satisfies 
\[
U^{n_1}\cdot Qx_1 = Q\cdot (U^{n_1} x^1_1 + Qx_0),
\]
we see that $U^{n_1}$ times the homology class of $Qx_1$ is contained in the image of $Q$ and cannot be annihilated by any powers of $U$. Therefore we get
\[
\bar{\upsilon}(K)-\ubar{\upsilon}(K) \ge (\mathbf{gr}(Qx_1)+1)-\mathbf{gr}(U^{n_1} x^1_1 + Qx_0) = n_1.\qedhere
\]
\end{proof}

We are finally ready to prove \Cref{thm:lowerbound}.

\begin{proof}[Proof of \Cref{thm:lowerbound}]
Since torus knots admit lens space surgeries \cite{moser}, and lens spaces are $L$-spaces \cite{ozsvath2005knot}, torus knots are $L$-space knots. The theorem follows directly from \Cref{mobiusthm}, \Cref{highergenusthm}, and \Cref{Lspaceknots}.
\end{proof}

\section{Asymptotically obstructing locally flat Möbius bands}\label{sec:asymp}
In this section we prove \Cref{asymptotic} by studying the linking forms of double branched covers of torus knots. It is well-known that double branched cover over $T_{p,q}$ is the Brieskorn sphere $\Sigma_2(S^3,T_{p,q})=\Sigma(2,p,q)$. Some number theory and a result of \cite{Yasuhara} allow us to show that the linking form obstructs certain torus knots $T_{p,q}$ --- subject to restrictions on $p$ and $q$ --- from bounding locally flat M\"obius bands. A subsequent argument shows that the restrictions on $p$ and $q$ are appropriately generic.

\subsection{The linking form of $\Sigma(2,p,q)$, where $p$ is even and $q$ is odd}
We begin by computing the linking forms of the Brieskorn sphere $\Sigma(2,p,q)$, for any even $p$ and odd $q$ such that $\gcd(p,q)=1$. Recall that a linking form of a rational homology 3-sphere $M$ is a $\mathbb{Q}/\mathbb{Z}$-valued bilinear form: 
\[
\lambda_M \colon H_1 (M;\mathbb{Z}) \times H_1 (M;\mathbb{Z}) \rightarrow \mathbb{Q}/\mathbb{Z}.
\]
To compute the linking form $\lambda_{\Sigma(2,p,q)}$ on $\Sigma(2,p,q)$, we start with a Goeritz matrix of $T_{p,q}$; recall that, in \cite[Section 3]{gordonlitherland}, it is shown that the linking form of the double cover of $S^3$ branched over a knot $K$ is presented by $-G^{-1} \text{mod } 1$, where $G$ is a Goeritz matrix of $K$. 

Consider the checkerboard surface $F$ for $T_{p,q}$, where $p$ is even, given in Figure \ref{fig:toruscheckerboard}. Note that $F$ is built from disks (0-handles) and bands (1-handles). Following \cite{gordonlitherland}, the Goeritz matrix $G$ of $T_{p,q}$ given by $F$ is the $\left(\frac{pq-2q+2}{2}\times\frac{pq-2q+2}{2}\right)$-matrix
\[
G=\begin{pmatrix}
q & -\mathbf{1}^h_q & 0 & 0 & \cdots & 0 & 0\\
-\mathbf{1}^v_q & A & -I_q & 0 & \cdots & 0 & 0\\
0 & -I_q & A & -I_q & \ddots & 0 & 0\\
0 & 0 & -I_q & A & \ddots & 0 & 0\\
\vdots & \vdots & \ddots & \ddots & \ddots & \vdots & \vdots \\
0 & 0 & 0 & 0 & \cdots & A & -I_q \\
0 & 0 & 0 & 0 & \cdots & -I_q & A
\end{pmatrix},
\]

where $I_q$ is the $q\times q$ identity matrix, $\mathbf{1}^h_q$ and $\mathbf{1}^v_q$ are the horizontal and vertical vectors of size $q$ whose entries are all $1$, respectively, and $A$ is the $(q\times q)$-matrix

\[
A=\begin{pmatrix}
0 & 1 & 0 & \cdots & 0 & 0 & 1 \\
1 & 0 & 1 & \cdots & 0 & 0 & 0 \\
0 & 1 & 0 & \cdots & 0 & 0 & 0 \\
\vdots  & \vdots  & \vdots & \ddots &\vdots & \vdots & \vdots\\
0 & 0 & 0 & \cdots & 0 & 1 & 0 \\
0 & 0 & 0 & \cdots & 1 & 0 & 1 \\
1 & 0 & 0 & \cdots & 0 & 1 & 0
\end{pmatrix}.
\]

\begin{figure}[htbp]
    \centering
    \includegraphics[width=0.7\textwidth]{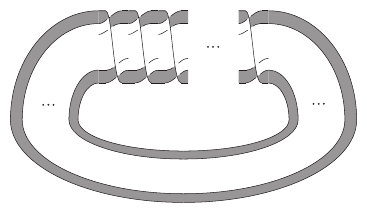}
    \caption{A checkerboard surface $F$ for the torus knot $T_{p,q}$, where $p$ is even.}
\label{fig:toruscheckerboard}
\end{figure}

Topologically, $G$ can be viewed as a matrix for the intersection form of the double cover of $B^4$ branched over $F$, where we consider $F$ as a properly embedded surface by pushing its interior into $B^4$. Denote this 4-manifold by $X$. Then $\partial X =\Sigma(2,p,q)$ and $X$ is simply connected. 
Instead of computing the full inverse of $G$, we will focus on two entries of $G^{-1}$, namely the $\left(1,\frac{pq-2q+2}{2}\right)$-entry and $\left(\frac{pq-2q+2}{2},\frac{pq-2q+2}{2}\right)$-entry. It will turn out that these two entries are enough to recover the linking form. Note that $\det(G)=\det(T_{p,q})=\vert \Delta_{T_{p,q}}(-1)\vert =q$ if $p$ is even and $q$ is odd. Since all entries of $G$ are integers, we thus see that all entries of $G^{-1}$ are integer multiples of $\frac{1}{q}$.

\begin{lem}
\label{entry1}
The $\left(1,\frac{pq-2q+2}{2}\right)$-entry of $G^{-1}$ is $\frac{1}{q}$.
\end{lem}
\begin{proof}
Consider the horizontal vector 
\[
v= \left( \frac{p}{2}, \hspace{.2cm} \underbrace{\frac{p}{2}-1, \cdots, \frac{p}{2}-1}_{q}, \hspace{.2cm} \underbrace{\frac{p}{2}-2, \cdots, \frac{p}{2}-2}_{q}, \hspace{.2cm} \cdots\cdots \hspace{.2cm} \underbrace{2, \cdots, 2}_{q}, \hspace{.2cm} \underbrace{1, \cdots, 1}_{q} \right) 
\]
of length $\frac{pq-2q+2}{2}$.
We claim that $vG=(q,0,\cdots,0)$. Denote the $k$th entry of $vG$ by $(vG)_k$. When $k=1$, we have 
\[
(vG)_1 = \frac{p}{2}\cdot q + q\cdot \left(\frac{p}{2}-1\right)\cdot (-1) = q.
\]
When $2\le k\le q+1$, we have
\[
(vG)_k = \frac{p}{2}\cdot (-1) + 2\cdot \left(\frac{p}{2}-1\right)\cdot 1 + \left(\frac{p}{2}-2\right)\cdot (-1)= 0.
\]
For $k\ge q+2$, it is easy to check that we also have $(vG)_k=0$. Thus $vG=(q,0,\ldots,0)$ and so $v$ is the top row of $qG^{-1}$. Therefore the $\left(1,\frac{pq-2q+2}{2}\right)$-entry of $G^{-1}$ is given by $\frac{1}{q}\cdot 1=\frac{1}{q}$.
\end{proof}

\begin{lem}
\label{entry2}
The $\left(\frac{pq-2q+2}{2},\frac{pq-2q+2}{2}\right)$-entry of $G^{-1}$ is of the form $\frac{m}{q}$, where $m$ is an integer satisfying $\frac{mp}{2} \equiv 1$ mod $q$.
\end{lem}

\begin{proof}
Let $v=\left(v_0,v_1^1,\cdots,v_1^q,v_2^1,\cdots,v_2^q,\cdots,v_{\frac{p}{2}-1}^1,\cdots,v_{\frac{p}{2}-1}^q\right)$ be the $i$th row of the matrix $q G^{-1}$. For simplicity, our notation of $v_k^r$ here is chosen to be cyclic on $r$ mod $q$, \ie $v_k^{r+q}=v_k^r$. Furthermore, we define $v_0^r=v_0$ for any $r$.

We will show that for any fixed $k\in \{1, \dots, \frac{p}{2}-1\}$, the integers $v_k^1,\cdots,v_k^q$ are congruent mod $q$. Since $v$ is the $i$th row of $qG^{-1}$, $vG$ is of the form $qe_{i}$, where $e_i$ is the $i$th standard basis vector. Thus every entry of $vG$ is either $0$ or $q$. Hence the following conditions hold:

\begin{itemize}
    \item For any indices $k$ and $r$ such that $0\le k \le \frac{p}{2}-3$, we have $v_{k+1}^{r}+v_{k+1}^{r+2}\equiv v_{k}^{r+1}+v_{k+2}^{r+1}$ mod $q$.
    \item For any index $r$, we have $v_{\frac{p}{2}-2}^{r+1} \equiv v_{\frac{p}{2}-1}^{r}+v_{\frac{p}{2}-1}^{r+2}$ mod $q$.
\end{itemize}

We claim that for any $r$ and $k$ with $0\le k\le \frac{p}{2}-1$, we have
\[
v_{k}^{r}+v_{k}^{r+2}+\cdots+v_{k}^{r+p-2k} =v_{k-1}^{r+1}+v_{k-1}^{r+3}+\cdots+v_{k-1}^{r+p-2k-1} \mod q.
\]
Using the second condition above as the base case, inductively assume that
\[
v_{k+1}^{r}+v_{k+1}^{r+2}+\cdots+v_{k+1}^{r+p-2k-2} \equiv v_{k}^{r+1}+v_{k}^{r+3}+\cdots+v_{k}^{r+p-2k-3} \mod q.
\]
If $\frac{p-2k}{2}$ is even, then by the first condition, the quantity 
\[
 v_{k}^{r}+v_{k}^{r+2}+v_{k}^{r+4}+v_{k}^{r+6}+\cdots+v_{k}^{r+p-2k-2}+v_{k}^{r+p-2k}
 \]
is equivalent modulo $q$ to 
\[
v_{k-1}^{r+1}+v_{k-1}^{r+5}+\cdots+v_{k-1}^{r+p+2k-7}+v_{k-1}^{r+p+2k-1} + v_{k+1}^{r+1}+v_{k+1}^{r+5}+\cdots+v_{k+1}^{r+p+2k-7}+v_{k+1}^{r+p+2k-1}.
\]
Applying the inductive hypothesis to the latter half of this expression, we see that this is equivalent modulo $q$ to 
\[
v_{k-1}^{r+1}+v_{k-1}^{r+5}+\cdots+v_{k-1}^{r+p+2k-7}+v_{k-1}^{r+p+2k-1}+v_k^{r+2}+v_k^{r+4}+\cdots +v_k^{r+p-2k-2}-v_{k+1}^{r+3}-v_{k+1}^{r+7}-\cdots-v_{k+1}^{r+p+2k-7}-v_{k+1}^{r+p+2k-1}
\]
Finally, applying the first condition to the terms $v_k^{r+2}+v_k^{r+4}+\cdots +v_k^{r+p-2k-2}$, the above expression is equivalent modulo $q$ to
\[
v_{k-1}^{r+1}+v_{k-1}^{r+3}+\cdots +v_{k-1}^{r+p-2k-1},
\]
as desired.

On the other hand, if $\frac{p-2k}{2}$ is odd, then by the first condition, the quantity 
\[
 v_{k}^{r}+v_{k}^{r+2}+v_{k}^{r+4}+v_{k}^{r+6}+\cdots+v_{k}^{r+p-2k-4}+v_{k}^{r+p-2k-4}+v_{k}^{r+p-2k}
 \]
is equivalent modulo $q$ to 
\[
 v_{k-1}^{r+1}+v_{k-1}^{r+5}+\cdots+v_{k-1}^{r+p+2k-7}+v_{k-1}^{r+p+2k-3} + v_{k+1}^{r+1}+v_{k+1}^{r+5}+\cdots+v_{k+1}^{r+p+2k-7}+v_{k+1}^{r+p+2k-3}+v_{k}^{r+p-2k}.
\]
Applying the inductive hypothesis to the latter half of the portion of the expression preceding $v_{k}^{r+p-2k}$, we see that this is equivalent modulo $q$ to 
\[
v_{k-1}^{r+1}+v_{k-1}^{r+5}+\cdots+v_{k-1}^{r+p+2k-7}+v_{k-1}^{r+p+2k-3}+v_k^{r+2}+v_k^{r+4}+\cdots +v_k^{r+p-2k}-v_{k+1}^{r+3}-v_{k+1}^{r+7}-\cdots-v_{k+1}^{r+p+2k-3}
\]
Finally, applying the first condition to the terms $v_k^{r+2}+v_k^{r+4}+\cdots +v_k^{r+p-2k}$, the above expression is equivalent modulo $q$ to
\[
v_{k-1}^{r+1}+v_{k-1}^{r+3}+\cdots +v_{k-1}^{r+p-2k-1},
\]
as desired.

Taking $k=1$ gives:
\[
v_{1}^{r}+v_{1}^{r+2}+\cdots+v_{1}^{r+p-2} = v_{0}^{r+1}+v_{0}^{r+3}+\cdots+v_{0}^{r+p-3} = \left(\frac{p}{2}-1\right) v_0\mod q.
\]
So, for any $r$, we have $v_{1}^{r}-v_{1}^{r+p}=(v_{1}^{r}+v_{1}^{r+2}+\cdots+v_{1}^{r+p-2})-(v_{1}^{r+2}+v_{1}^{r+4}+\cdots+v_{1}^{r+p})=0$ mod $q$, \ie $v_{1}^{r}=v_{1}^{r+p}$ mod $q$. But since $q$ is relatively prime to $p$, there exists an integer $\alpha$ such that $\alpha p=1$ mod $q$, so we get $v_{1}^{r}=v_{1}^{r+\alpha p}=v_{1}^{r+1}$ mod $q$ for all $r$. Therefore we deduce that $v_{1}^{1}=\cdots=v_{1}^{q}$ mod $q$. It now follows by the first condition and induction on $k$, that for a fixed $k$, the integers $v_k^1,\ldots,v_k^q$ are congruent modulo $q$.

We next claim that for all $k$ and $r$, $v_{k}^{r} \equiv \left(\frac{p}{2}-k\right)v_{\frac{p}{2}-1}^{q}$ mod $q$. We will prove this via induction and appealing to the fact that for all $k$, $v_k^1,\ldots,v_k^q$ are equivalent mod $q$. The claim is clearly true when $k=\frac{p}{2}-1$ and $k=\frac{p}{2}-2$. Assume inductively that the claim is true for all $j\ge k+1$. Then 
$$v_{k}^{r} \equiv v_{k+1}^{r-1}+v_{k+1}^{r+1}-v_{k+2}^r \equiv 2v_{k+1}^r-v_{k+2}^r\equiv  2\left(\frac{p}{2}-k-1\right)v_{\frac{p}{2}-1}^{q}-\left(\frac{p}{2}-k-2\right)v_{\frac{p}{2}-1}^{q}\equiv\left(\frac{p}{2}-k\right)v_{\frac{p}{2}-1}^{q}\mod q.$$ 

In particular, when $k=0$, we have that $v_0\equiv \frac{p}{2}v_{\frac{p}{2}-1}^q$ mod $q$. Thus, letting $v$ be the last row of $qG^{-1}$, we have $(qG^{-1})_{\frac{pq-2q+2}{2},1}\equiv\frac{p}{2}\cdot (qG^{-1})_{\frac{pq-2q+2}{2},\frac{pq-2q+2}{2}}$ mod $q$.
Since $G^{-1}$ is symmetric, we have that $(qG^{-1})_{1,\frac{pq-2q+2}{2}}\equiv\frac{p}{2}\cdot (qG^{-1})_{\frac{pq-2q+2}{2},\frac{pq-2q+2}{2}}$ mod $q$.
By \Cref{entry1}, the $(1,\frac{pq-2q+2}{2})$-entry of $G^{-1}$ is $\frac{1}{q}$. Therefore, $(qG^{-1})_{\frac{pq-2q+2}{2},1}=(qG^{-1})_{1,\frac{pq-2q+2}{2}}=1$ and so $\frac{p}{2}\cdot (qG^{-1})_{\frac{pq-2q+2}{2},\frac{pq-2q+2}{2}}\equiv 1$ mod $q$. Thus we now see that $(qG^{-1})_{\frac{pq-2q+2}{2},\frac{pq-2q+2}{2}}$ is congruent to the modular inverse of $\frac{p}{2}$ mod $q$, as desired.
\end{proof}
To calculate $\lambda_{\Sigma(2,p,q)}$ explicitly, we need to understand the domain on which the pairing is defined, namely $H_1 (\Sigma_2(S^3,T_{p,q});\mathbb{Z})$. It turns out that this group is always cyclic, so that $\lambda_{\Sigma(2,p,q)}$ is determined by a single rational number.
\begin{lem}
$H_1 (\Sigma_2 (S^3,T_{p,q});\mathbb{Z})\cong \mathbb{Z}/q\mathbb{Z}$.
\label{lem:H1}
\end{lem}
\begin{proof}
The Alexander module $A_{T_{p,q}}$ of $T_{p,q}$ is given by $\mathbb{Z}[t,t^{-1}]/(\Delta(t))$, where $\Delta(t)$ denotes the Alexander polynomial of $T_{p,q}$; this can be shown easily using Fox calculus. It is known that taking its quotient by the ideal generated by $t+1$ gives the first homology group $H_1$ of $\Sigma_2(S^3,T_{p,q})$; see \cite[Chapter 4]{stevens1996homology} for a general formula for $H_1$ of branched cyclic covers over knots. Therefore we get
\[
H_1 (\Sigma_2(S^3,T_{p,q});\mathbb{Z})\simeq A_{T_{p,q}} /(t+1) \simeq \mathbb{Z}[t,t^{-1}]/(t+1,\Delta_{T_{p,q}}(t)) \simeq \mathbb{Z}/\Delta_{T_{p,q}}(-1)\mathbb{Z} \simeq \mathbb{Z}/\det(T_{p,q})\mathbb{Z}.
\]
Since $\det(T_{p,q})=q$, we deduce that $H_1 (\Sigma_2(S^3,T_{p,q});\mathbb{Z})\simeq \mathbb{Z}/q\mathbb{Z}$.
\end{proof}

Now we are ready to compute $\lambda_{\Sigma(2,p,q)}$.

\begin{prop}
\label{linkingform}
For any relatively prime positive integers $p$ and $q$ with $p$ even, we have $\lambda_{\Sigma(2,p,q)}(x,x) \equiv - \frac{p}{2q}$ mod $1$ for some generator $x$ of $H_1(\Sigma(2,p,q);\mathbb{Z})$.
\end{prop}
\begin{proof}
As above, let $X$ denote the double cover of $B^4$ branched over the surface $F$ after it is pushed into $B^4$. Then $\partial X=\Sigma(2,p,q)$ and $X$ is simply connected. Consequently, we have the following long exact sequence 
$$0\to H_2(X)\xrightarrow{G} H_2(X,\Sigma(2,p,q))\xrightarrow{f} H_1(\Sigma(2,p,q))\to 0.$$
Given $y\in H_1(\Sigma(2,p,q))$, to compute $\lambda_{\Sigma(2,p,q)}(y,y)$ using $G$, we must first lift $y$ to $H_2(X,\Sigma(2,p,q))$ via the map $f$.
Note that we can view $H_1(\Sigma(2,p,q))$ as the cokernel of $G$ and $f$ as the quotient map. 

Since the $\left(1,\frac{pq-2q+2}{2}\right)$-entry of $G^{-1}$ is $\frac{1}{q}$, the element $(0,\dots,0,1)\in H_2 (X,\Sigma(2,p,q))$ is a lift of $1\in H_1 (\Sigma(2,p,q))\cong \mathbb{Z}/q\mathbb{Z}$. To see this, let $k$ be any integer satisfying $f(0,\dots,0,k)=0$. Then we have $(0,\dots,0,k)^T-Gv=0$ for some vector $v=(v_1,\dots,v_{\frac{pq-2q+2}{2}})^T$, where all $v_i$ are integers. But then $v=G^{-1} (0,\dots,0,k)^T$, so by comparing the first entries of both sides, we get $v_1=\frac{k}{q}$. Since $v_1\in\mathbb{Z}$, $k$ must be a multiple of $q$. Thus we see that $f(0,\dots,0,1)$ generates $H_1(\Sigma(2,p,q))$, as desired.

Now, recall that the linking form $\lambda_{\Sigma(2,p,q)}$ is presented by $-G^{-1}$ mod $1$. Thus the value of  $-\lambda_{\Sigma(2,p,q)}(1,1)=(0,\dots,0,1)G^{-1}(0,\dots,0,1)^T$ is given by the $\left(\frac{pq-2q+2}{2},\frac{pq-2q+2}{2}\right)$-entry of $G^{-1}$, which by \Cref{entry2} is given by $\frac{m}{q}$ where $m$ is some integer satisfying $\frac{mp}{2} \equiv 1$ mod $q$. So we get $\lambda_{\Sigma(2,p,q)}\left(\frac{p}{2},\frac{p}{2}\right)=-\frac{mp^2}{4q} \equiv -\frac{p}{2q}$ mod $1$. Since $q$ is relatively prime to $p$, $\frac{p}{2}\in H_1(\Sigma(2,p,q))$ is also a generator of $H_1(\Sigma(2,p,q))$. Therefore taking $x=\frac{p}{2}$ proves the lemma.
\end{proof}

\begin{rem}
Feller and Golla proved in \cite{feller2020nonorientable} that the linking form of the Brieskorn sphere $\Sigma(2,2p,2pk\pm 1)$ represents $\mp\frac{p}{2pk+1}$ as a square. This is coherent with our general computation in \Cref{linkingform}.
\end{rem}

\subsection{Number Theory Interlude}
Before we can prove \Cref{asymptotic}, we need some number theoretic results.

\begin{lem}
\label{numlem}
Suppose that $p$ is even and $\frac{p}{2}$ is not a perfect square. There exists an integer $r$, relatively prime to $p$, such that for any prime $s$ satisfying $s\equiv r$ mod $2p$, the integers $\frac{p}{2}$ and $-\frac{p}{2}$ are not quadratic residues mod $s$.
\end{lem}

To prove this, we need the following lemma.

\begin{lem}[{\cite[Lemma 1.14]{cox2011primes}}]
\label{lem1}
Let $n$ be an integer such that $n\equiv 0,1 \pmod 4$. Then there exists a unique group homomorphism 
\[
\chi_n \colon (\mathbb{Z}/n\mathbb{Z})^\times \rightarrow \{\pm 1\}
\]
such that $\left(\frac{n}{s}\right)=\chi_n (s)$ for all odd primes $s$ not dividing $n$. Here, $\left(\frac{n}{s}\right)$ denotes the Legendre symbol. Furthermore, we have 
\[
\chi_n (-1) = \begin{cases}
1 &\text{ if } n>0,\\
-1 &\text{ if } n<0.
\end{cases}
\]
\end{lem}

\begin{proof}[Proof of \Cref{numlem}]
Since $\frac{p}{2}$ is not a perfect square, $2p$ is a multiple of $4$ which is not a perfect square.
Using \Cref{lem1}, there is a unique nontrivial group homomorphism $\chi_{2p}\colon (\mathbb{Z}/2p\mathbb{Z})^\times \rightarrow \{\pm 1\}$ such that $\left(\frac{2p}{s}\right)=\chi_{2p} (s)$ for all odd primes $s$ not dividing $p$. It is a well-known fact (see, for example, \cite[Theorem 3 of Section 5.2]{ireland1990classical}) that for any non-square integer $n$, there exist infinitely many primes for which $n$ is a quadratic nonresidue; since $2p$ is not a perfect square, we see that $\chi_{2p}$ cannot be a trivial map. Hence the size of $\ker\left(\chi_{2p}\right)$ is exactly half of its domain $(\mathbb{Z}/2p\mathbb{Z})^\times$. 

We claim that there exists an element $r$ not contained in $\ker\left(\chi_{2p}\right)$ such that $r\equiv1 \pmod 4$.
Suppose not. Then for each $s\notin \ker\left(\chi_{2p}\right)$, $s\equiv 3 \pmod 4$. Since half of the elements in $(\mathbb{Z}/2p\mathbb{Z})^\times$ are of the form $4n+3$ and the other half are of the form $4n+1$, we have that $\ker\chi_{2p}=\{s\in(\mathbb{Z}/2p\mathbb{Z})^{\times}\,|\,s\equiv 1\pmod{4}\}$. However, we have $\chi_{2p}(-1)=1$ since $2p>0$, while $-1 \equiv 3 \mod 4$. Hence, we have reached a contradiction and there must exist an element $r\notin\ker(\chi_{2p})$ such that $r\equiv1\pmod4$.

Now, for any prime $s$ congruent to $r \mod 2p$, we have $-1=\chi_{2p}(r)=\left( \frac{2p}{s}\right)=\left(\frac{2}{s}\right)^2 \left(\frac{p/2}{s}\right) = \left(\frac{p/2}{s}\right)$. Moreover, since $s\equiv 1\pmod 4$, we have that $\left(\frac{-1}{s}\right)=1$ and, consequently, $\left(\frac{-p/2}{s}\right)=\left(\frac{-1}{s}\right)\left(\frac{p/2}{s}\right)=-1$. Therefore both $\frac{p}{2}$ and $-\frac{p}{2}$ are quadratic nonresidues modulo $s$.
\end{proof}

\subsection{Proof of \Cref{asymptotic}}
Recall the following theorem of Murakami and Yasuhara, which is a direct corollary of their lower bound on $\gammatop$ from \Cref{lem:linkingpairing}.

\begin{thm}[{\cite[Corollary 2.7]{murakami2000four}}]\label{thm:murakami}
If a knot $K$ in $S^3$ bounds a locally flat M{\" o}bius band in $B^4$, then the linking form $\lambda$ on $H_1 (\Sigma_2(S^3,K);\mathbb{Z})$ splits into a direct sum $(G_1 ,\lambda_1) \oplus (G_2 ,\lambda_2)$, where $\lambda_1$ is represented by the $(1 \times 1)$-matrix $\left(\pm \frac{1}{\vert G_1 \vert}\right)$ and $\lambda_2$ is metabolic.
\end{thm}

To state the next lemma, we introduce the following definition.

\begin{defn}
Let $n$ be a nonzero integer and let $p$ be a prime factor of $n$. Let $r$ be the maximal positive integer such that $p^r$ divides $n$. When $r$ is odd, we say that $p$ is an \emph{odd-power prime factor} of $n$.
\end{defn}

Now using \Cref{thm:murakami}, we can prove the following lemma. 

\begin{lem}
\label{obstlem}
Let $K$ be a knot in $S^3$ satisfying $H_1 (\Sigma_2(S^3,K);\mathbb{Z})\cong \mathbb{Z}/n\mathbb{Z}$ for some positive integer $n$. Write the linking form $\lambda$ of $\Sigma_2(S^3,K)$ as $\lambda(x,x)=\frac{m}{n}x^2 \mod 1$ for some integer $m$. Suppose that there exists an odd-power prime factor $s$ of $n$ such that $m$ and $-m$ are not quadratic residues mod $s$. Then $K$ does not bound a locally flat M{\" o}bius band in $B^4$.
\end{lem}

\begin{proof}
Suppose that $K$ bounds a M{\" o}bius band. Then we have a splitting 
$$(H_1 (\Sigma_2(S^3,K);\mathbb{Z}),\lambda)=(G_1,\lambda_1)\oplus (G_2,\lambda_2),$$
where $\lambda_1$ is represented by a $(1 \times 1)$-matrix $\left(\pm \frac{1}{\vert G_1 \vert}\right)$ and $\lambda_2$ is metabolic. The order of $G_2$ is a perfect square, so if we write $G_1 =\mathbb{Z}/k\mathbb{Z}$ and $G_2 = \mathbb{Z}/\ell \mathbb{Z}$, then $\ell$ is a perfect square and $s$ divides $k$.

Now let $1$ be a generator of $H_1 (\Sigma_2(S^3,K);\mathbb{Z})$. Then $G_1$ is generated by $\ell=\ell\cdot 1$. By our assumption on $\lambda_1$, there exists an integer $r$, relatively prime to $k$, such that $\lambda_1 (r\ell,r\ell)=\pm \frac{1}{k}$. But since $\lambda_1 (r\ell,r\ell)=\lambda(r\ell,r\ell)=\frac{m\ell r^2}{k}$ and $\ell$ is a perfect square, either $m$ or $-m$ must be a quadratic residue mod $k$. Since $s$ divides $k$, this implies that either $m$ or $-m$ is a quadratic residue mod $s$, a contradiction.
\end{proof}

\begin{exmp} Consider the torus knots $T_{4,q}$, where $q$ is odd. By \Cref{linkingform}, the linking form on $\Sigma_2(S^3,T_{4,q})$ is given by $\lambda(x,x)=-\frac{2}{q}\mod 1$ for some generator $x$ of $H_1(\Sigma_2(S^3,T_{4,q});\mathbb{Z})$. It is well-known that $\pm2$ are both quadratic nonresidues modulo $a$ if and only if $a\equiv 5\pmod 8$. Thus, in this case, the only possible value of $r$ in the statement of Lemma \ref{numlem} is $r=5$. Now by Lemma \ref{obstlem}, if $q$ has a odd-power prime factor $r\equiv 5\pmod 8$, then $T_{4,q}$ does not bound a locally flat M{\"o}bius band.
\label{exmp:4q}
\end{exmp}

More generally, we can prove the following
\begin{prop}
\label{torusobstlem}
Let $p,q>0$ be relatively prime integers, where $p$ is even. Suppose that $\frac{p}{2}$ is not a perfect square. Then there exists an integer $r$, relatively prime to $p$, such that $T_{p,q}$ does not bound a locally flat M{\" o}bius band in $B^4$ whenever $q$ admits an odd-power prime divisor congruent to $r$ mod $2p$.
\end{prop}
\begin{proof}
By \Cref{numlem}, there exists an integer $r$, relatively prime to $p$, such that for any prime $s$ congruent to $r$ mod $2p$, $\frac{p}{2}$ and $-\frac{p}{2}$ are quadratic nonresidues mod $s$. The lemma then follows directly from \Cref{obstlem}.
\end{proof}

For the proof of \Cref{asymptotic}, we will use the following estimate.
\begin{lem}
\label{primelem}
Let $P$ be a set consisting of primes, such that the infinite sum $\sum\limits_{s\in P} \frac{1}{s}$ diverges to $\infty$. Given a positive integer $N$, denote the set of positive integers $n$ less than $N$ whose odd-exponent prime factors are not contained in $P$ by $S_{P,N}$, and the set of positive integers $n$ less than $N$ which are relatively prime to $p$ as $T_{p,N}$. Then we have 
\[
\lim_{N\rightarrow\infty} \frac{\vert S_{P,N}\cap T_{p,N}\vert}{\vert T_{p,N} \vert} =0.
\]
\end{lem}
\begin{proof}
Without loss of generality, we can assume that $P$ does not contain any prime factors of $p$. Then, by the inclusion-exclusion principle, the given limit can be computed as the infinite product
\[
\prod_{s\in P} \left(1-\frac{1}{s}+\frac{1}{s^2}-\frac{1}{s^3}+\cdots\right)=\prod_{s\in P} \frac{s}{s+1}.
\]
Taking logarithms gives
\[
\log\left(\prod_{s\in P} \frac{s}{s+1}\right)=-\sum_{s\in P} (\log(s+1)-\log(s)) \le -\sum\limits_{s\in P} \frac{1}{s}=-\infty.
\]
Here, we used the fact that $\log(s+1)-\log(s)\ge \frac{1}{s}$ and $\sum\limits_{s\in P}\frac{1}{s}=\infty$. Therefore the limit given in the statement of the lemma converges to zero.
\end{proof}

Recall the following estimate on the sum of reciprocals of primes in arithmetic progressions, which is a well-known corollary of the prime number theorem.
\begin{lem}
\label{distlem}
For any choice of positive integers $N$, $k$, and $\ell$, such that $\ell$ is relatively prime to $k$, there exists a constant $C_{k,\ell}$ which depends only on $k$ and $\ell$ such that
\[
\sum_{\substack{s\le N,\,s\equiv \ell \pmod k,\\ s\text{ is prime}}} \frac{1}{s} = \frac{\log\log N}{\phi(k)} + C_{k,\ell} + O\left(\frac{1}{\log N}\right).
\]
\end{lem}

We are finally ready to prove the desired asymptotics.
\begin{proof}[Proof of \Cref{asymptotic}]
By \Cref{torusobstlem}, there exists some $r$, relatively prime to $p$, such that $T_{p,q}$ does not bound a locally flat M{\" o}bius band in $B^4$ whenever $q$ admits an odd-power prime divisor congruent to $r$ mod $2p$. Moreover, by \Cref{distlem}, the sum of reciprocals of all primes congruent to $r$ mod $2p$ diverges to $\infty$. The theorem then follows from \Cref{primelem}.
\end{proof}

\bibliographystyle{amsalpha}
\bibliography{citations}

\providecommand{\bysame}{\leavevmode\hbox to3em{\hrulefill}\thinspace}
\providecommand{\MR}{\relax\ifhmode\unskip\space\fi MR }
\providecommand{\MRhref}[2]{%
  \href{http://www.ams.org/mathscinet-getitem?mr=#1}{#2}
}
\providecommand{\href}[2]{#2}
\begin{thebibliography}{Zem19b}

\bibitem[All23]{allen2020nonorientable}
Samantha Allen, \emph{Nonorientable surfaces bounded by knots: a geography
  problem}, New York J. Math. \textbf{29} (2023), 1038--1059. \MR{4646147}

\bibitem[Bal20]{ballinger2020concordance}
William Ballinger, \emph{Concordance invariants from the {$ E (-1) $} spectral
  sequence on khovanov homology}, arXiv preprint arXiv:2004.10807 (2020).

\bibitem[Bat14]{batson2012nonorientable}
Joshua Batson, \emph{Nonorientable slice genus can be arbitrarily large}, Math.
  Res. Lett. \textbf{21} (2014), no.~3, 423--436. \MR{3272020}

\bibitem[BCG17]{bodnarceloriagolla}
J\'{o}zsef Bodn\'{a}r, Daniele Celoria, and Marco Golla, \emph{A note on
  cobordisms of algebraic knots}, Algebr. Geom. Topol. \textbf{17} (2017),
  no.~4, 2543--2564. \MR{3686406}

\bibitem[BHS19]{borodzik2019involutive}
Maciej Borodzik, Jennifer Hom, and Andrzej Schinzel, \emph{Involutive
  {H}eegaard {F}loer homology and rational cuspidal curves}, Proceedings of the
  London Mathematical Society \textbf{118} (2019), no.~3, 441--472.

\bibitem[BL14]{borodziklivingston2014}
Maciej Borodzik and Charles Livingston, \emph{Heegaard {F}loer homology and
  rational cuspidal curves}, Forum Math. Sigma \textbf{2} (2014), Paper No.
  e28, 23. \MR{3347955}

\bibitem[Cla78]{clark}
Bradd~Evans Clark, \emph{Crosscaps and knots}, Internat. J. Math. Math. Sci.
  \textbf{1} (1978), no.~1, 113--123. \MR{478131}

\bibitem[Cox11]{cox2011primes}
David~A Cox, \emph{Primes of the form $x^2+ ny^2$: {F}ermat, class field
  theory, and complex multiplication}, vol.~34, John Wiley \& Sons, 2011.

\bibitem[DM19]{dai2019involutive}
Irving Dai and Ciprian Manolescu, \emph{Involutive {H}eegaard {F}loer homology
  and plumbed three-manifolds}, Journal of the Institute of Mathematics of
  Jussieu \textbf{18} (2019), no.~6, 1115--1155.

\bibitem[Don87]{donaldson}
S.~K. Donaldson, \emph{The orientation of {Y}ang-{M}ills moduli spaces and
  {$4$}-manifold topology}, J. Differential Geom. \textbf{26} (1987), no.~3,
  397--428. \MR{910015}

\bibitem[DS24]{daemi2020chern}
Aliakbar Daemi and Christopher Scaduto, \emph{Chern--{S}imons functional,
  singular instantons, and the four-dimensional clasp number}, J. Eur. Math.
  Soc. (JEMS) \textbf{26} (2024), no.~6, 2127--2190. \MR{4742808}

\bibitem[Fan19]{fan2019unoriented}
Haofei Fan, \emph{Unoriented cobordism maps on link {F}loer homology}, [PhD
  Thesis, University of California, Los Angeles], 2019.

\bibitem[FG23]{feller2020nonorientable}
Peter Feller and Marco Golla, \emph{Non-orientable slice surfaces and inscribed
  rectangles}, Ann. Sc. Norm. Super. Pisa Cl. Sci. (5) \textbf{24} (2023),
  no.~3, 1463--1485. \MR{4675965}

\bibitem[FK17]{fellerkrcatovich}
Peter Feller and David Krcatovich, \emph{On cobordisms between knots, braid
  index, and the upsilon-invariant}, Math. Ann. \textbf{369} (2017), no.~1-2,
  301--329. \MR{3694648}

\bibitem[FPR19]{fellerparkray}
Peter Feller, JungHwan Park, and Arunima Ray, \emph{On the {U}psilon invariant
  and satellite knots}, Math. Z. \textbf{292} (2019), no.~3-4, 1431--1452.
  \MR{3980298}

\bibitem[Fr{\o}02]{froyshov}
Kim~A. Fr{\o}yshov, \emph{Equivariant aspects of {Y}ang-{M}ills {F}loer
  theory}, Topology \textbf{41} (2002), no.~3, 525--552. \MR{1910040}

\bibitem[Gha22]{ghanbarian2020}
Nakisa Ghanbarian, \emph{The non-orientable 4-genus for knots with 10
  crossings}, J. Knot Theory Ramifications \textbf{31} (2022), no.~5, Paper No.
  2250034, 46. \MR{4450931}

\bibitem[GL78]{gordonlitherland}
C.~McA. Gordon and R.~A. Litherland, \emph{On the signature of a link}, Invent.
  Math. \textbf{47} (1978), no.~1, 53--69. \MR{500905}

\bibitem[GL11]{gilmerlivingston}
Patrick~M. Gilmer and Charles Livingston, \emph{The nonorientable 4-genus of
  knots}, J. Lond. Math. Soc. (2) \textbf{84} (2011), no.~3, 559--577.
  \MR{2855790}

\bibitem[GLM81]{gordonlitherlandmurasugi}
C.~McA. Gordon, R.~A. Litherland, and K.~Murasugi, \emph{Signatures of covering
  links}, Canadian J. Math. \textbf{33} (1981), no.~2, 381--394. \MR{617628}

\bibitem[GM18]{gollamarengon}
Marco Golla and Marco Marengon, \emph{Correction terms and the nonorientable
  slice genus}, Michigan Math. J. \textbf{67} (2018), no.~1, 59--82.
  \MR{3770853}

\bibitem[GM23]{gong2021non}
Sherry Gong and Marco Marengon, \emph{Nonorientable link cobordisms and torsion
  order in {F}loer homologies}, Algebr. Geom. Topol. \textbf{23} (2023), no.~6,
  2627--2672. \MR{4640136}

\bibitem[Hed09]{hedden2009}
Matthew Hedden, \emph{On knot {F}loer homology and cabling. {II}}, Int. Math.
  Res. Not. IMRN (2009), no.~12, 2248--2274. \MR{2511910}

\bibitem[HL17]{hogancamp2017involutive}
Matthew Hogancamp and Charles Livingston, \emph{An involutive upsilon knot
  invariant}, arXiv preprint arXiv:1710.08360 (2017).

\bibitem[HM17]{hendricks2017involutive}
Kristen Hendricks and Ciprian Manolescu, \emph{Involutive {H}eegaard {F}loer
  homology}, Duke Mathematical Journal \textbf{166} (2017), no.~7, 1211--1299.

\bibitem[IR90]{ireland1990classical}
Kenneth Ireland and Michael Rosen, \emph{A classical introduction to modern
  number theory}, Graduate Texts in Mathematics, Springer, 1990.

\bibitem[JK18]{jabukakelly}
Stanislav Jabuka and Tynan Kelly, \emph{The nonorientable 4-genus for knots
  with 8 or 9 crossings}, Algebr. Geom. Topol. \textbf{18} (2018), no.~3,
  1823--1856. \MR{3784020}

\bibitem[JTZ21]{juhasz2018naturality}
Andr\'{a}s Juh\'{a}sz, Dylan Thurston, and Ian Zemke, \emph{Naturality and
  mapping class groups in {H}eegard {F}loer homology}, Mem. Amer. Math. Soc.
  \textbf{273} (2021), no.~1338, v+174. \MR{4337438}

\bibitem[Juh16]{juhasz2016cobordisms}
Andr{\'a}s Juh{\'a}sz, \emph{Cobordisms of sutured manifolds and the
  functoriality of link {F}loer homology}, Advances in Mathematics \textbf{299}
  (2016), 940--1038.

\bibitem[JVC20]{jabukavanCott3and4}
Stanislav Jabuka and Cornelia~A. Van~Cott, \emph{Comparing nonorientable three
  genus and nonorientable four genus of torus knots}, J. Knot Theory
  Ramifications \textbf{29} (2020), no.~3, 2050013, 15. \MR{4101607}

\bibitem[JVC21]{jabuka2019nonorientable}
\bysame, \emph{On a nonorientable analogue of the {M}ilnor conjecture}, Algebr.
  Geom. Topol. \textbf{21} (2021), no.~5, 2571--2625. \MR{4334520}

\bibitem[JZ20]{juhasz2020concordance}
Andr{\'a}s Juh{\'a}sz and Ian Zemke, \emph{New {H}eegaard {F}loer slice genus
  and clasp number bounds}, arXiv preprint arXiv:2007.07106 (2020).

\bibitem[KM93]{kronheimermrowka}
Peter~B Kronheimer and Tomasz~S Mrowka, \emph{Gauge theory for embedded
  surfaces, {I}}, Topology \textbf{32} (1993), no.~4, 773--826.

\bibitem[Lev66]{levine}
Jerome Levine, \emph{Polynomial invariants of knots of codimension two}, Annals
  of Mathematics (1966), 537--554.

\bibitem[Lob19]{lobb}
Andrew Lobb, \emph{A counterexample to {B}atson's conjecture}, Math. Res. Lett.
  \textbf{26} (2019), no.~6, 1789. \MR{4078695}

\bibitem[Lon20]{longo}
Vincent Longo, \emph{An infinite family of counterexamples to {B}atson's
  conjecture}, arXiv preprint arXiv:2011.00122 (2020).

\bibitem[MO07]{manolescu2007concordance}
Ciprian Manolescu and Brendan Owens, \emph{A concordance invariant from the
  {F}loer homology of double branched covers}, International Mathematics
  Research Notices \textbf{2007} (2007), no.~9, rnm077--rnm077.

\bibitem[Mos71]{moser}
Louise Moser, \emph{Elementary surgery along a torus knot}, Pacific J. Math.
  \textbf{38} (1971), 737--745. \MR{383406}

\bibitem[MY00]{murakami2000four}
Hitoshi Murakami and Akira Yasuhara, \emph{Four-genus and four-dimensional
  clasp number of a knot}, Proceedings of the American Mathematical Society
  \textbf{128} (2000), no.~12, 3693--3699.

\bibitem[NW15]{niwu}
Yi~Ni and Zhongtao Wu, \emph{Cosmetic surgeries on knots in {$S^3$}}, J. Reine
  Angew. Math. \textbf{706} (2015), 1--17. \MR{3393360}

\bibitem[OS03]{ozsvath2003absolutely}
Peter Ozsv{\'a}th and Zolt{\'a}n Szab{\'o}, \emph{Absolutely graded {F}loer
  homologies and intersection forms for four-manifolds with boundary}, Advances
  in Mathematics \textbf{173} (2003), no.~2, 179--261.

\bibitem[OS05]{ozsvath2005knot}
\bysame, \emph{On knot {F}loer homology and lens space surgeries}, Topology
  \textbf{44} (2005), no.~6, 1281--1300.

\bibitem[OSS17]{ozsvath2017unoriented}
Peter~S Ozsv{\'a}th, Andr{\'a}s~I Stipsicz, and Zolt{\'a}n Szab{\'o},
  \emph{Unoriented knot {F}loer homology and the unoriented four-ball genus},
  International Mathematics Research Notices \textbf{2017} (2017), no.~17,
  5137--5181.

\bibitem[Ras04]{rasmussen2004lens}
Jacob Rasmussen, \emph{Lens space surgeries and a conjecture of {G}oda and
  {T}eragaito}, Geometry \& Topology \textbf{8} (2004), no.~3, 1013--1031.

\bibitem[Ste96]{stevens1996homology}
Wayne~H Stevens, \emph{On the homology of branched cyclic covers of knots},
  [PhD Thesis, Louisiana State University and Agricultural \& Mechanical
  College] (1996).

\bibitem[Tai20]{tairi}
Nafie Tairi, \emph{The smooth nonorientable 4-genus of an infinite family of
  torus knots}, [Master's Thesis, University of Bern] (2020).

\bibitem[Ter04]{teragaito}
Masakazu Teragaito, \emph{Crosscap numbers of torus knots}, Topology Appl.
  \textbf{138} (2004), no.~1-3, 219--238. \MR{2035482}

\bibitem[Vir75]{viro}
Oleg~Yanovich Viro, \emph{Positioning in codimension 2, and the boundary},
  Uspekhi Matematicheskikh Nauk \textbf{30} (1975), no.~1, 231--232.

\bibitem[Wal04]{wall2004singular}
Charles Terence~Clegg Wall, \emph{Singular points of plane curves}, no.~63,
  Cambridge University Press, 2004.

\bibitem[Wan18]{wang2018semigroups}
Shida Wang, \emph{Semigroups of {L}-space knots and nonalgebraic iterated torus
  knots}, Mathematical Research Letters \textbf{25} (2018), no.~1, 335--346.

\bibitem[Yas96]{Yasuhara}
Akira Yasuhara, \emph{Connecting lemmas and representing homology classes of
  simply connected {$4$}-manifolds}, Tokyo J. Math. \textbf{19} (1996), no.~1,
  245--261. \MR{1391941}

\bibitem[Zem19a]{zemke2019connected}
Ian Zemke, \emph{Connected sums and involutive knot {F}loer homology},
  Proceedings of the London Mathematical Society \textbf{119} (2019), no.~1,
  214--265.

\bibitem[Zem19b]{zemke2019link}
\bysame, \emph{Link cobordisms and functoriality in link {F}loer homology},
  Journal of Topology \textbf{12} (2019), no.~1, 94--220.

\end{thebibliography}
\end{document}